\documentclass[12pt]{amsart}
\usepackage[colorlinks=true]{hyperref}
\usepackage[all]{xy}
\usepackage{amsmath,amssymb,amsthm}
\usepackage{txfonts}
\usepackage{enumerate}
\usepackage{graphicx}
\usepackage[figuresright]{rotating}
\theoremstyle{plain}
	\newtheorem{thm}{Theorem}[section]
	\newtheorem{prp}[thm]{Proposition}
	\newtheorem{lem}[thm]{Lemma}
	\newtheorem{cor}[thm]{Corollary}
\theoremstyle{definition}
	\newtheorem{dfn}[thm]{Definition}
	\newtheorem{ex}[thm]{Example}
\theoremstyle{remark}
	\newtheorem{rem}[thm]{Remark}

\newcommand{ \Fix}{\operatorname{Fix}}
\newcommand{ \Eq}{\operatorname{Eq}}
\newcommand{ \Hofix}{\operatorname{Hofix}}
\newcommand{ \Hoeq}{\operatorname{Hoeq}}

\newcommand{ \im}{\operatorname{im}}
\newcommand{ \id}{\operatorname{id}}

\newcommand{ \ind}{\operatorname{ind}}

\newcommand{ \tr}{\operatorname{tr}}
\newcommand{ \Geom}{\operatorname{Geom}}

\begin{document}
\title{Coincidence Reidemeister trace and its generalization}
\author{Mitsunobu Tsutaya}
\address{Faculty of Mathematics, Kyushu University, Fukuoka 819-0395, Japan}
\email{tsutaya@math.kyushu-u.ac.jp}
\date{}
\begin{abstract}
We give a homotopy invariant construction of the Reidemeister trace for the coincidence of two maps between closed manifolds of not necessarily the same dimensions.
It is realized as a homology class of the homotopy equalizer, which coincides with the Hurewicz image of Koschorke's stabilized bordism invariant.
To define it, we use a kind of shriek maps appearing string topology.
As an application, we compute the coincidence Reidemeister trace for the self-coincidence of the projections of $S^1$-bundles on $\mathbb{C}P^n$.
We also mention how to relate our construction to the string topology operation called the loop coproduct.
\end{abstract}
\maketitle

\section{Introduction}

For a continuous self-map $f\colon X\to X$, a \textit{fixed point} is a point $x\in X$ such that $f(x)=x$.
Topological fixed point theory has been studied for more than 100 years.
One of the most remarkable result is the \textit{Lefschetz fixed point theorem}:

\textit{For a self map $f\colon M\to M$ on a compact $m$-dimensional manifold $M$, if the alternating sum of the traces
\[
	\lambda(f)=\sum_{i=0}^m(-1)^i\tr(H_\ast(f;\mathbb{Q}))
\]
is non-zero, then there exists a fixed point of $f$.}

This number is called the \textit{Lefschetz trace} or the \textit{Lefschetz number}.
Of course, it is a homotopy invariant of $f$.
Though the Lefschetz trace is easy to compute in many cases, it often vanishes even if $f$ has unremovable fixed points.
The \textit{Reidemeister trace} $\rho(f)$ is a refinement of $\lambda(f)$ given as an element of the free abelian group generated by the set of certain equivalence classes of paths on $M$, which is a homotopy invariant as well.
Under the obvious augmentation, $\rho(f)$ is mapped to $\lambda(f)$.

As a similar problem, let us consider the coincidence of two maps $f,g\colon M\to N$ between closed oriented manifolds.
That is, what can we say about the \textit{equalizer} (or the \textit{coincidence set})
\[
	\Eq(f,g):=\{x\in M\mid f(x)=g(x)\}?
\]
In fact, the Lefschetz-type coincidence theorem is known for $M$, $N$ of not necessarily the same dimensions.
See \cite{MR1853657}.
Then, how about the Reidemeister trace?
This is the main theme of the present paper.

Actually, we will define the \textit{coincidence Reidemeister trace} as follows.
Consider the homotopy pullback square
\[
\xymatrix{
	M \ar[d]_-{(f,g)}
		& \Hoeq(f,g) \ar[l]_-{\tilde{\Delta}} \ar[d] \\
	N\times N
		& N, \ar[l]^-{\Delta}
}
\]
where $\Delta$ is the diagonal map.
The space $\Hoeq(f,g)$ is called the \textit{homotopy equalizer}.
For such an diagram, using the construction appearing in string topology \cite{MR1942249}, we have the map of reverse direction on homology:
\[
	\tilde{\Delta}^!\colon H_\ast(M)\to H_{\ast-n}(\Hoeq(f,g)).
\]
Then, our coincidence Reidemeister trace $\rho(f,g)$ is defined to be the image of the fundamental class
\[
	\rho(f,g)=\tilde{\Delta}^![M]\in H_{m-n}(\Hoeq(f,g)).
\]
We will describe the geometric meaning of this invariant.

From a stable homotopy point of view, we have the Pontrjagin--Thom map between Thom spectra
\[
	R(f,g)\colon M^{-TM}\to\Hoeq(f,g)^{TN-TM},
\]
where $-TM$ denotes the stable normal bundle of $M$ and $\Hoeq(f,g)^{TN-TM}$ the Thom spectrum for the direct sum of the bundles pulled back by the canonical projections onto $M$ and $N$.
We call $R(f,g)$ the \textit{Reidemeister map} of $f$ and $g$.
Composing the unit map $S^0\to M^{-TM}$, we obtain the element of the stable homotopy group $\rho^{\pi}(f,g)\in\pi_{0}(\Hoeq(f,g)^{TN-TM})$, which we also call the coincidence Reidemeister trace.

In fact, this is not a completely new invariant.
For fixed point problem, it has already been known by Klein--Williams \cite[Section 10]{MR2326939} that the Reidemeister trace is obtained from the above procedure.
For coincidence problem, Koschorke \cite{MR2270573} defined the stabilized bordism invariant by a sum of \textit{local} Reidemeister traces in some sense, which coincides with the Reidemeister trace $\rho^\pi(f,g)$.
Note that the homotopy group $\pi_{0}(\Hoeq(f,g)^{TN-TM})$ is identified with the bordism group $\Omega_{m-n}(\Eq(f,g);TN-TM)$ of Hatcher--Quinn \cite{MR0353322}.
Ponto \cite[Section 4]{MR3463529} has already obtained the map $M_+\to\Hoeq(f,g)^{TN}$, which coincides with our Reidemeister map after taking the Thom spectrum of the stable normal bundle $-TM$.

The above map $\tilde{\Delta}$ is also known to induce the map on Serre spectral sequence by Cohen--Jones--Yan \cite{MR2039760}.
Then the Serre spectral sequence can be applied to compute the Reidemeister trace.
This is one of the advantage of our Reidemeister trace though the Koschorke's stabilized bordism invariant is stronger than it.
As an example, we study the self-coincidences of the projections of principal $S^1$-bundles over $\mathbb{C}P^n$.

This paper is organized as follows.
In Section \ref{fixed}, we recall the Lefschetz and Reidemeister traces for fixed points.
In Section \ref{geometric coincidence}, following the previous section, we recall the coincidence Lefschetz trace and define the coincidence Reidemeister trace geometrically.
In Section \ref{shriek maps}, we recall the shriek maps in the form used in string topology.
In Section \ref{invariant traces}, we define the local Reidemeister trace using the shriek map, and prove that it satisfies a variant of axioms of Reidemeister trace by Staecker \cite{MR2529499}.
As a corollary, our geometric Reidemeister trace coincides with the \textit{global} Reidemeister trace.
In Section \ref{properties}, we prove some more properties of the coincidence Reidemeister trace.
In Section \ref{Thom spectra}, we generalize our construction for related Thom spectra.
This clarifies the relations among the works by Koschorke, by Klein--Williams, and by Ponto.
In Section \ref{Nielsen numbers}, we state the relation among various generalizations of Nielesen number.
In Section \ref{shriek map on ss}, we recall the shriek map induced on the Serre spectral sequence.
In Section \ref{example1}, using the Serre spectral sequence, we compute the Reidemeister trace for the self-coincidence of the projections of $S^1$-bundles on $\mathbb{C}P^n$.
In Section \ref{loop coproduct}, we remark on the relation between our Reidemeister trace and the loop coproduct in string topology.

The author would like to thank Professor Ponto for letting him know the related works.

In the whole of this paper, we follow the sign convention in Dold's book \cite{MR1335915}.

\section{Geometric description of Lefschetz and Reidemeister traces for fixed points}
\label{fixed}

We denote the integral homology by $H_\ast$.

First of all, we recall the fixed point index.
Let $\varphi\colon\mathbb{R}^m\to\mathbb{R}^m$ be a continuous map such that $f(x)=x$ if and only if $x=0$.
Then we have the map
\[
	\Phi\colon(\mathbb{R}^m,\mathbb{R}^m-0)\to(\mathbb{R}^m,\mathbb{R}^m-0),\qquad x\mapsto x-f(x)
\]
and hence the induced homomorphism
\[
	\Phi_\ast\colon H_m(\mathbb{R}^m,\mathbb{R}^m-0)\to H_m(\mathbb{R}^m,\mathbb{R}^m-0).
\]
The \textit{fixed point index} $\ind(\varphi;0)\in\mathbb{Z}$ is defined by
\[
	\Phi_\ast a=\ind(\varphi;0)a.
\]

Next, we recall the Lefschetz trace.
Let $f\colon M\to M$ be a self-map on a connected closed oriented manifold $M$.
Deforming $f$ if necessary, we assume that the set of fixed points
\[
	\Fix(f)=\{x\in M\mid f(x)=x\}
\]
is finite.
For $x\in\Fix(f)$, take a neighborhood $D_x$ and $D'_x$ of $x$ such that the following conditions are satisfied:
\begin{itemize}
\item
there are orientation-preserving homeomorphisms $(D_x,x)\cong(\mathbb{R}^m,0)$ and $(D_x',x)\cong(\mathbb{R}^m,0)$,
\item
$\Fix(f)\cap D_x=\{x\}$,
\item
$f(D_x)\cup D_x\subset D'_x$.
\end{itemize}
Then we have the map $\varphi_x\colon\mathbb{R}^m\to\mathbb{R}^m$ defined by the composite
\[
	\mathbb{R}^m\cong D_x\xrightarrow{f}D'_x\cong\mathbb{R}^m.
\]
Note that $\varphi_x(y)=y$ if and only if $y=0$.
\begin{dfn}
Under the above notation, the \textit{fixed point index} $\ind(f;x)$ of $f$ at $x$ is defined by
\[
	\ind(f;x)=\ind(\varphi_x;0).
\]
\end{dfn}
Note that $\ind(f;x)$ is independent of the choice of the neighborhoods $D_x,D'_x$ and the homeomorphisms $D_x\cong\mathbb{R}^m,D'_x\cong\mathbb{R}^m$. 
\begin{dfn}
The \textit{Lefschetz trace} $\lambda(f)\in\mathbb{Z}$ of $f$ is defined by
\[
	\lambda(f)=\sum_{x\in\Fix(f)}\ind(f;x).
\]
\end{dfn}

In fact, the Lefschetz trace is homotopy invariant by the following \textit{Lefschetz fixed point theorem}.
\begin{thm}[Lefschetz]
Let $f\colon M\to M$ be a self-map on a closed oriented connected manifold $M$.
Then the following equality holds:
\[
	\lambda(f)=\sum_{i=0}^m(-1)^i\tr(H_i(f;\mathbb{Q})).
\]
\end{thm}

We also recall the Reidemeister trace.
\begin{dfn}
The \textit{homotopy fixed point} $\Hofix(f)$ of $f$ is defined by
\[
	\Hofix(f)=\{(x,\ell)\in M\times M^I\mid f(\ell(1))=\ell(0)=x\},
\]
where $M^I$ is the space of paths $I=[0,1]\to M$.
\end{dfn}
There is the canonical projection $\pi\colon\Hofix(f)\to M$ given by $\pi(x,\ell)=x$.
For $x\in\Fix(f)$, we denote the homotopy class of the constant path at $x$ by $[x]\in\pi_0(\Hofix(f))$.
\begin{dfn}
The \textit{Reidemeister trace} $\rho(f)\in\mathbb{Z}[\pi_0(\Hofix(f))]$ is defined by
\[
	\rho(f)=\sum_{x\in\Fix(f)}\ind(f;x)[x].
\]
\end{dfn}

The Reidemeiseter trace is known to be a homotopy invariant.
Note that the free abelian group $\mathbb{Z}[\pi_0(\Hofix(f))]$ is naturally isomorphic to $H_0(\Hofix(f))$.
The Reidemeister trace has better information than the Lefschetz trace because
\[
	\pi_\ast\rho(f)=\lambda(f)[\ast]\in H_0(M)
\]
where $[\ast]$ is the homology class represented by the basepoint $\ast\in M$.

The elements in $\pi_0(\Hofix(f))$ are called \textit{Reidemeister classes}.
For a fixed point class $a\in\pi_0(\Hofix(f))$, if the coefficient of $a$ in the Reidemeister trace $\rho(f)$ is nonzero, the fixed point class $a$ is said to be \textit{essential}.
Of course, the subset of essential fixed point classes is homotopy invariant.
The cardinality of this set is known as the classical \textit{Nielsen number}, which is of course a homotopy invariant.

\section{Coincidence Lefschetz and Reidemeister traces of different dimensions}
\label{geometric coincidence}

In the rest of this paper, we will denote the homology of coefficients in a unital commutative ring $R$ by $H_\ast$.
We say an $m$-dimensional connected closed manifold $M$ is \textit{oriented} if $H_m(M)\cong R$ and the fundamental class $[M]\in H_m(M)$ is given.
Under this convention, every connected closed manifold is oriented with the unique fundamental class if $R=\mathbb{Z}/2\mathbb{Z}$.

Following the argument in the previous section, we recall the coincidence Lefschetz trace and introduce the coincidence Reidemeister trace.
Let $f,g\colon M\to N$ be maps between smooth connected closed oriented manifolds $M$ and $N$.
We denote $\dim M=m$ and $\dim N=n$ and allow the case when $m$ and $n$ are different.
But if $m<n$, one can deform $f$ and $g$ so that there are no $x\in M$ such that $f(x)=g(x)$.
Thus we assume $m\ge n$.
The \textit{equalizer} $\Eq(f,g)$ of $f$ and $g$ is defined by
\[
	\Eq(f,g)=\{x\in M\mid f(x)=g(x)\}.
\]
This subset is also called the \textit{coincidence set}.
We assume that $\Eq(f,g)\subset M$ is a smooth closed $(m-n)$-dimensional submanifold.
By applying the Thom transversality theorem, this is the case for generic deformations of $f$ and $g$.
We decompose $\Eq(f,g)$ as the disjoint union of the connected components:
\[
	\Eq(f,g)=L_1\sqcup\cdots\sqcup L_k.
\]
We may assume that $L_1,\ldots,L_{k_0}$ are orientable and the rest are not.
We fix the orientations of $L_1,\ldots,L_{k_0}$.
Take the tubular neighborhood $V_j\subset M$ of $L_j$, which is identified with the normal bundle of $L_j$.
Let $v_j\in H_{m-n}(V_j,V_j-L_j)$ be the corresponding Thom class, which is characterized by the property that under the composite
\[
	H_m(M)\to H_m(M,M-L_j)\cong H_m(V_j,V_j-L_j)\xrightarrow{v_j\frown}H_{m-n}(V_j)\cong H_{m-n}(L_j),
\]
the fundamental class $[M]$ is mapped to $[L_j]$.

We define the coincidence index of orientable $L_j$ as follows.
Take $x\in L_j$ and neighborhoods $D_x$ of $x$ and $D'_{f(x)}$ of $f(x)$ such that
\begin{enumerate}
\item
there are homeomorphisms $(D_x,D_x\cap L_j)\cong(\mathbb{R}^m,\mathbb{R}^{m-n}\times\{0\})$ and $(D'_{f(x)},f(x))\cong(\mathbb{R}^n,\{0\})$ which respect the orientations,
\item
$L_i\cap D_x=\emptyset$ for $i\ne j$,
\item
$f(D_x)\cup g(D_x)\subset D'_{f(x)}$.
\end{enumerate}
Then we have the maps $\varphi_x,\psi_x\colon\mathbb{R}^n\to\mathbb{R}^n$ defined by the composites
\begin{align*}
	\varphi_x\colon&\mathbb{R}^n\xrightarrow{(0,\id_{\mathbb{R}^n})}\mathbb{R}^m\cong D_x\xrightarrow{f}D'_x\cong\mathbb{R}^n,\\
	\psi_x\colon&\mathbb{R}^n\xrightarrow{(0,\id_{\mathbb{R}^n})}\mathbb{R}^m\cong D_x\xrightarrow{g}D'_x\cong\mathbb{R}^n.
\end{align*}
Note that $\varphi_x(y)=\psi_x(y)$ if and only if $y=0$.
We define the map
\[
	\Phi_x\colon(\mathbb{R}^n,\mathbb{R}^n-\{0\})\to(\mathbb{R}^n,\mathbb{R}^n-\{0\}),\qquad y\mapsto\psi_x(y)-\varphi_x(y).
\]
\begin{dfn}
Under the above notation, the \textit{coincidence index} $\ind(f,g;L_j)\in\mathbb{Z}$ of $f$ on $L_j$ is the integer such that the induced map
\[
	(\Phi_x)_\ast\colon H_n(\mathbb{R}^n,\mathbb{R}^n-\{0\})\to H_n(\mathbb{R}^n,\mathbb{R}^n-\{0\})
\]
satisfies $(\Phi_x)_\ast a=\ind(f,g;L_j)a$.
\end{dfn}
Note that $\ind(f,g;L_j)$ is independent of the choice of $x\in L_j$, the neighborhoods $D_x,D'_{f(x)}$ and the homeomorphisms $D_x\cong\mathbb{R}^m,D'_{f(x)}\cong\mathbb{R}^n$.
\begin{dfn}
The \textit{geometric coincidence Lefschetz trace} $\lambda^{\Geom}(f,g)\in H_{m-n}(N)$ of $f$ and $g$ is defined by
\[
	\lambda^{\Geom}(f,g)=\sum_{j=1}^{k_0}\ind(f,g;L_j)f_\ast[L_j].
\]
\end{dfn}

\begin{rem}
\label{rem_another_Lefschetz}
The above definition follows Saveliev \cite{MR1853657}.
The sum
\[
\sum_{j=1}^{k_0}\ind(f,g;L_j)[L_j]\in H_{m-n}(M)
\]
is another generalization of Lefschetz trace.
Of course, its image under the homomorphism $f_\ast$ or $g_\ast$ is $\lambda^{\Geom}(f,g)$.
Basic property of this invariant will be discussed in Corollaries \ref{cor_another_Lefschetz} and \ref{cor_primary_obstruction}.
\end{rem}

To define the coincidence Reidemeister trace, we recall the homotopy equalizer.
\begin{dfn}
The \textit{homotopy equalizer} $\Hoeq(f,g)$ of $f$ and $g$ is defined by
\[
	\Hoeq(f,g)=\{(x,\ell)\in M\times N^I\mid \ell(0)=f(x),\ell(1)=g(x)\}.
\]
We denote the evaluation of the path at $1/2$ by $\pi\colon\Hoeq(f,g)\to N$.
\end{dfn}
There is the canonical map
\[
	\iota\colon\Eq(f,g)\to\Hoeq(f,g),\qquad x\mapsto(x,(\text{the constant path at }f(x)=g(x))).
\]
Then we have the homology class $\iota_\ast[L_j]\in H_{m-n}(\Hoeq(f,g))$.
Using such classes, we define the coincidence Reidemeister trace as follows:
\begin{dfn}
The \textit{geometric coincidence Reidemeister trace} $\rho^{\Geom}(f,g)\in H_{m-n}(\Hoeq(f,g))$ is defined by
\[
	\rho^{\Geom}(f,g)=\sum_{j=1}^k\ind(f,g;L_j)\iota_\ast[L_j].
\]
\end{dfn}

We note that $\pi_\ast\rho^{\Geom}(f,g)=\lambda^{\Geom}(f,g)$.
In Section \ref{invariant traces}, we will construct $\rho(f,g)$ in a homotopy invariant way and prove that $\rho(f,g)=\rho^{\Geom}(f,g)$.

\section{Shriek maps and Thom spectra}
\label{shriek maps}

In this section, we recall the classical Pontrjagin--Thom construction of manifolds and the generalization for homotopy pullback.
Most of the constructions we will give are also found in \cite{MR1942249,MR2326939}.

Let $\delta\colon X_1\to X_2$ be a smooth map between smooth connected closed oriented manifolds.
We denote the dimensions by $d_1=\dim X_1$ and $d_2=\dim X_2$ and the codimension by $d=d_2-d_1$ which is not necessarily nonnegative.
Consider a pullback square
\[
\xymatrix{
	E_1 \ar[r]^{\tilde{\delta}} \ar[d]_-{\pi_1}
		& E_2 \ar[d]^-{\pi_2} \\
	X_1 \ar[r]_-{\delta}
		& X_2,
}
\]
where $\pi_1$ and $\pi_2$ are fibrations.

For a sufficiently large $r$, there is an embedding $i\colon X_1\hookrightarrow D^r$ into the interior of the $r$-dimensional unit disk $D^r$.
Then the map
\[
	\delta'\colon X_1\xrightarrow{(i,\delta)}D^r\times X_2
\]
is an embedding.
Take a tubular neighborhood $U\subset(D^r-S^{r-1})\times X_2$ of $\delta'(X_1)$.
The following square is again a pullback square:
\[
\xymatrix{
	E_1 \ar[r]^-{\tilde{\delta}'} \ar[d]_-{\pi_1}
		& D^r\times E_2 \ar[d]^-{\id\times\pi_2} \\
	X_1 \ar[r]_-{\delta'}
		& D^r\times X_2,
}
\]
where $\tilde{\delta}':=(i\circ\pi_1,\tilde{\delta})$.
For the neighborhood $\tilde{U}:=(\id\times\pi_2)^{-1}U$ of $\tilde{\delta}'(E_1)$, the pair $(\tilde{U},\tilde{U}-\tilde{\delta}'(E_1))$ is homotopy equivalent to the pair $(\pi_1^{-1}\nu,\pi_1^{-1}\nu-E_1)$ by the homotopy lifting property of $\id\times\pi_2$, where $\nu$ is the normal bundle of $\delta'(X_1)$.
Note that the following construction on homology is independent of the choice of $r$, $i$ and $U$.
\begin{dfn}
\label{dfn_shriek}
The \textit{shriek maps}
\begin{align*}
	\delta^!&\colon H_i(X_2)\to H_{i-d}(X_1),\\
	\tilde{\delta}^!&\colon H_i(E_2)\to H_{i-d}(E_1)
\end{align*}
are defined by the composites
\begin{align*}
	\delta^!\colon &H_i(X_2)
		\cong H_{i+r}(D^r\times X_2,S^{r-1}\times X_2)\\
		&\to H_{i+r}(D^r\times X_2,D^r\times X_2-\delta'(X_1))
		\cong H_{i+r}(U,U-\delta'(X_1))
		\xrightarrow{u\frown}H_{i-d}(U)
		\cong H_{i-d}(X_1),\\
	\tilde{\delta}^!\colon &H_i(E_2)
		\cong H_{i+r}(D^r\times E_2, S^{r-1}\times E_2)\\
		&\to H_{i+r}(D^r\times E_2,D^r\times E_2-\tilde{\delta}'(E_1))
		\cong H_{i+r}(\tilde{U},\tilde{U}-\tilde{\delta}'(X_1))
		\xrightarrow{\pi_2^\ast u\frown}H_{i-d}(\tilde{U})
		\cong H_{i-d}(E_1),
\end{align*}
where $u\in H_{r+d}(U,U-\delta'(X_1))$ is the Thom class of the normal bundle of $\delta'(X_1)$.
\end{dfn}

\begin{rem}
The Thom class $u\in H_{i+r}(U,U-\delta'(X_1))$ is characterized by the following property: under the composite
\[
	H_{i+r}(D^r\times X_2,S^{r-1}\times X_2)\\
		\to H_{i+r}(D^r\times X_2,D^r\times X_2-\delta'(X_1))
		\cong H_{i+r}(U,U-\delta'(X_1))
		\xrightarrow{u\frown}H_{i-d}(U)
		\cong H_{i-d}(X_1),
\]
the cross product $g_r\times[X_2]$ is mapped to the fundamental class $[X_1]$, where $g_r\in H_r(D^r,S^{r-1})$ is the generator respecting the orientation.
\end{rem}

The following proposition immediately follows from the construction.
\begin{prp}
\label{prp_shriek_property1}
\begin{enumerate}
\item
The shriek map $\delta^!$ maps the fundamental class of $X_2$ to that of $X_1$, that is, $\delta^![X_2]=[X_1]$.
\item
The following diagram commutes:
\[
\xymatrix{
	H_i(E_2) \ar[r]^-{\tilde{\delta}^!} \ar[d]_-{(\pi_2)_\ast}
		& H_{i-d}(E_1) \ar[d]^-{(\pi_1)_\ast} \\
	H_i(X_2) \ar[r]_-{\delta^!}
		& H_{i-d}(X_1).	
}
\]
\end{enumerate}
\end{prp}

Let us consider the homotopy pullback $E$ of the maps
\[
	X_1\xrightarrow{\delta_1}X_3\xleftarrow{\delta_2}X_2
\]
of smooth connected closed oriented manifolds.
We denote the dimensions by $d_i=\dim X_i$.
Let $\epsilon_i\colon X_3^I\to X_3$ the evaluation at $i$, i.e., $\epsilon_i(\ell)=\ell(i)$.
Consider the following pullback squares:
\[
\xymatrix{
	E \ar[r]^-{\tilde{\delta}_2} \ar[d]^-{\tilde{\delta}_1}
		& \tilde{X}_1 \ar[r]^-{\epsilon_0} \ar[d]^-{\tilde{\delta}_1}
		& X_1 \ar[d]^-{\delta_1} \\
	\tilde{X}_2 \ar[r]^-{\tilde{\delta}_2} \ar[d]^-{\epsilon_1}
		& X_3^I \ar[r]^-{\epsilon_0} \ar[d]^-{\epsilon_1}
		& X_3 \\
	X_2 \ar[r]^-{\delta_2}
		& X_3.
}
\]
These squares are homotopy pullback as well.
We denote the obvious inclusions by $i_1\colon X_1\to\tilde{X}_1$ and $i_2\colon X_2\to\tilde{X}_2$.

\begin{prp}
\label{prp_pullback_shriek}
The following square commutes up to the sign $(-1)^{(d_3-d_1)(d_3-d_2)}$:
\[
\xymatrix{
	H_i(X_3) \ar[r]^-{\delta_1^!} \ar[d]_-{\delta_2^!}
		& H_{i-d_3+d_2}(X_1) \ar[d]^-{\tilde{\delta}_2^!\circ(i_1)_\ast} \\
	H_{i-d_3+d_1}(X_2) \ar[r]_-{\tilde{\delta}_1^!\circ(i_2)_\ast}
		& H_{i-2d_3+d_1+d_2}(E).
}
\]
\end{prp}

\begin{proof}
Take embeddings $X_1\hookrightarrow D^{r_1}$ and $X_2\hookrightarrow D^{r_2}$ as before.
Let $U_i\subset D^{r_i}\times X_3$ be a tubular neighborhood of $X_i$ and $\tilde{U}_i\subset D^{r_i}\times X_3^I$ its inverse image by $\id\times\epsilon_i$.
We have the pullback squares
\[
\xymatrix{
	E \ar[r]^-{\tilde{\delta}'_2} \ar[d]^-{\tilde{\delta}'_1}
		& \tilde{X}_1\times D^{r_2} \ar[r]^-{\epsilon_0\times\id} \ar[d]^-{\tilde{\delta}'_1}
		& X_1\times D^{r_2} \ar[d]^-{\delta'_1} \\
	D^{r_1}\times\tilde{X}_2 \ar[r]^-{\tilde{\delta}'_2} \ar[d]^-{\id\times\epsilon_1}
		& D^{r_1}\times X_3^I\times D^{r_2} \ar[r]^-{\id\times\epsilon_0\times\id} \ar[d]^-{\id\times\epsilon_1\times\id}
		& D^{r_1}\times X_3\times D^{r_2} \\
	D^{r_1}\times X_2 \ar[r]^-{\delta'_2}
		& D^{r_1}\times X_3\times D^{r_2}.
}
\]
Interchanging the factors as $D^{r_2}\times X_3\cong X_3\times D^{r_2}$ and $D^{r_2}\times X_3^I\cong X_3^I\times D^{r_2}$, we denote the corresponding tubular neighborhoods by $U_2'$ and $\tilde{U}_2'$, respectively.
We also have the following tubular neighborhoods of $E$:
\begin{itemize}
\item
$D^{r_1}\times\tilde{X}_2\cap\tilde{U}_1\times D^{r_2}$ is a tubular neighborhood of $E\subset D^{r_1}\times\tilde{X}_2$,
\item
$\tilde{X}_1\times D^{r_2}\cap D^{r_1}\times\tilde{U}_2'$ is a tubular neighborhood of $E\subset\tilde{X}_1\times D^{r_2}$,
\item
$\tilde{U}_{12}=\tilde{U}_1\times D^{r_1}\cap D^{r_1}\times\tilde{U}_2'$ is a tubular neighborhood of $E\subset D^{r_1}\times X_3^I\times D^{r_2}$, which is naturally homotopy equivalent to the Whitney sum of the restriction of the normal bundles of $\tilde{X}_1\times D^{r_2}$ and of $D^{r_1}\times\tilde{X}_2$.
\end{itemize}
We denote the Thom classes by
\begin{align*}
	&u_i\in H^{r_i+c_i}(U_i,U_i-X_i),&&u_2'\in H^{r_2+c_2}(U_2',U_2'-X_2),\\
	&\tilde{u}_i\in H^{r_i+c_i}(U_i,U_i-X_i),&&\tilde{u}_2'\in H^{r_2+c_2}(U_2',U_2'-X_2),
\end{align*}
where we denote the codimension by $c_i=d_3-d_i$.
Note that, under the homeomorphisms $(U_2,U_2-X_2)\cong(U_2',U_2'-X_2)$ and $(\tilde{U}_2,\tilde{U}_2-X_2)\cong(\tilde{U}_2',\tilde{U}_2'-\tilde{X}_2)$, the Thom classes $u_2$ and $\tilde{u}_2$ correspond to $(-1)^{d_3r_2}u_2'$ and $(-1)^{d_3r_2}\tilde{u}_2'$, respectively. 

Then we have the following commutative diagrams, where $j_1\colon\tilde{U}_{12}\to\tilde{U}_1\times D^{r_2}$ and $j_2\colon\tilde{U}_{12}\to D^{r_2}\times\tilde{U}_2'$ are inclusions, $\mathcal{D}^{r_i}=(D^{r_i},S^{r_i-1})$ and the signs inside squares mean that the corresponding squares commute up to the indicated signs:
\[
\xymatrix{
	H_i(X_3^I) \ar@{=}[r] \ar[d]^-{g_{r_2}\times} \ar@{}[dr]|{(-1)^{ir_2}}
		& H_i(X_3^I) \ar[d]^-{\times g_{r_2}} \\
	H_{i+r_2}((D^{r_2},S^{r_2-1})\times X_3^I) \ar[r]^-{\cong} \ar[d]
		& H_{i+r_2}(X_3^I\times(D^{r_2},S^{r_2-1})) \ar[d] \\
	H_{i+r_2}(D^{r_2}\times X_3^I,D^{r_2}\times X_3^I-\tilde{X}_2) \ar[r]^-{\cong}
		& H_{i+r_2}(X_3^I\times D^{r_2},X_3^I\times D^{r_2}-\tilde{X}_2) \\
	H_{i+r_2}(\tilde{U}_2,\tilde{U}_2-\tilde{X}_2) \ar[r]^-{\cong} \ar[u]_-{\cong} \ar[d]^-{\tilde{u}_2\frown} \ar@{}[dr]|{(-1)^{d_3r_2}}
		& H_{i+r_2}(\tilde{U}_2',\tilde{U}_2'-\tilde{X}_2) \ar[u]_-{\cong} \ar[d]^-{\tilde{u}_2'\frown} \\
	H_{i-c_2}(\tilde{U}_2) \ar[r]^-{\cong}
		& H_{i-c_2}(\tilde{U}_2') \\
	H_{i-c_2}(\tilde{X}_2) \ar@{=}[r] \ar[u]_-{\cong}
		&H_{i-c_2}(\tilde{X}_2), \ar[u]_-{\cong}
}
\]
\begin{turn}{90}
\centering
\begin{tabular}{c} 
$\displaystyle{
\xymatrix{
	H_i(X_3^I) \ar[r]^-{g_{r_1}\times}
			\ar[d]^-{\times g_{r_2}}
		& H_{i+r_1}(\mathcal{D}^{r_1}\times X_3^I)
			\ar[r] \ar[d]^-{\times g_{r_2}}
		& H_{i+r_1}(D^{r_1}\times X_3^I,D^{r_1}\times X_3^I-\tilde{X}_1)
			\ar[d]^-{\times g_{r_2}} \\
	H_{i+r_2}(X_3^I\times\mathcal{D}^{r_2})
			\ar[r]^-{g_{r_1}\times} \ar[d]
		& H_{i+r_1+r_2}(\mathcal{D}^{r_1}\times X_3^I\times\mathcal{D}^{r_2})
			\ar[r] \ar[d]
		& H_{i+r_1+r_2}((D^{r_1}\times X_3^I,D^{r_1}\times X_3^I-\tilde{X}_1)\times\mathcal{D}^{r_2})
			\ar[d] \\
	H_{i+r_2}(X_3^I\times D^{r_2},X_3^I\times D^{r_2}-\tilde{X}_2)
			\ar[r]^-{g_{r_1}\times} 
		& H_{i+r_1+r_2}(\mathcal{D}^{r_1}\times(X_3^I\times D^{r_2},X_3^I\times D^{r_2}-\tilde{X}_2))
			\ar[r]
		& H_{i+r_1+r_2}(D^{r_1}\times X_3^I\times D^{r_2},D^{r_1}\times X_3^I\times D^{r_2}-E) \\
	H_{i+r_2}(\tilde{U}'_2,\tilde{U}_2'-\tilde{X}_2)
			\ar[r]^-{g_{r_1}\times} \ar[u]_-{\cong} \ar[d]^-{\tilde{u}_2\frown} \ar@{}[dr]|{(-1)^{r_1(c_2+r_2)}}
		& H_{i+r_1+r_2}(\mathcal{D}^{r_1}\times(\tilde{U}'_2,\tilde{U}_2'-\tilde{X}_2))
			\ar[r] \ar[u]_-{\cong} \ar[d]^-{(1\times\tilde{u}_2)\frown}
		& H_{i+r_1+r_2}(D^{r_1}\times\tilde{U}_2',D^{r_1}\times\tilde{U}_2'-E)
			\ar[u]_-{\cong} \ar[d]^-{(1\times\tilde{u}_2)\frown} \\
	H_{i-c_2}(\tilde{U}'_2)
			\ar[r]^-{g_{r_1}\times}
		& H_{i+r_1-c_2}(\mathcal{D}^{r_1}\times\tilde{U}'_2)
			\ar[r]
		& H_{i+r_1-c_2}(D^{r_1}\times\tilde{U}_2',D^{r_1}\times\tilde{U}_2'-\tilde{X}_1) \\
	H_{i-c_2}(\tilde{X}_2)
			\ar[r]^-{g_{r_1}\times} \ar[u]_-{\cong}
		& H_{i+r_1-c_2}(\mathcal{D}^{r_1}\times\tilde{X}_2)
			\ar[r] \ar[u]_-{\cong}
		& H_{i+r_1-c_2}(D^{r_1}\times\tilde{X}_2,D^{r_1}\times\tilde{X}_2-\tilde{X}_1) \ar[u]_-{\cong}
}
}$\\
\\
$\displaystyle{
\xymatrix{
	H_{i+r_1}(D^{r_1}\times X_3^I,D^{r_1}\times X_3^I-\tilde{X}_1) \ar[d]^-{\times g_{r_2}}
		& H_{i+r_1}(\tilde{U}_1,\tilde{U}_1-\tilde{X}_1) \ar[l]_-{\cong} \ar[r]^-{\tilde{u}_1\frown} \ar[d]^-{\times g_{r_2}}
		& H_{i-c_1}(\tilde{U}_1) \ar[d]^-{\times g_{r_2}} \\
	H_{i+r_1+r_2}((D^{r_1}\times X_3^I,D^{r_1}\times X_3^I-\tilde{X}_1)\times\mathcal{D}^{r_2}) \ar[d]
		& H_{i+r_1+r_2}((\tilde{U}_1,\tilde{U}_1-\tilde{X}_1)\times\mathcal{D}^{r_2}) \ar[l]_-{\cong} \ar[r]^-{(\tilde{u}_1\times1)\frown} \ar[d]
		& H_{i-c_1+r_2}(\tilde{U}_1\times\mathcal{D}^{r_2}) \ar[d] \\
	H_{i+r_1+r_2}(D^{r_1}\times X_3^I\times D^{r_2},D^{r_1}\times X_3^I\times D^{r_2}-E)
		& H_{i+r_1+r_2}(\tilde{U}_1\times D^{r_2},\tilde{U}_1\times D^{r_2}-E) \ar[l]_-{\cong} \ar[r]^-{(\tilde{u}_1\times1)\frown}
		& H_{i-c_1+r_2}(\tilde{U}_1\times D^{r_2},\tilde{U}_1\times D^{r_2}-\tilde{X}_2) \\
	H_{i+r_1+r_2}(D^{r_1}\times\tilde{U}_2',D^{r_1}\times\tilde{U}_2'-E)
			\ar[d]^-{(1\times\tilde{u}_2)\frown} \ar[u]_-{\cong}
		& H_{i+r_1+r_2}(\tilde{U}_{12},\tilde{U}_{12}-E)
			\ar[l]_-{\cong} \ar[r]^-{j_1^\ast(\tilde{u}_1\times1)\frown}
			\ar[d]^-{j_2^\ast(1\times\tilde{u}_2)\frown} \ar[u]_-{\cong} \ar@{}[dr]|{(-1)^{(c_1+r_1)(c_2+r_2)}}
		& H_{i-c_1+r_2}(\tilde{U}_{12},\tilde{U}_{12}-\tilde{X}_2)
			\ar[d]^-{j_2^\ast(1\times\tilde{u}_2)\frown} \ar[u]_-{\cong} \\
	H_{i+r_1-c_2}(D^{r_1}\times\tilde{U}_2',D^{r_1}\times\tilde{U}_2'-\tilde{X}_1)
		& H_{i+r_1-c_2}(\tilde{U}_{12},\tilde{U}_{12}-\tilde{X}_1) \ar[l]_-{\cong} \ar[r]^-{j_1^\ast(\tilde{u}_1\times1)\frown}
		& H_{i-c_1-c_2}(\tilde{U}_{12},\tilde{U}_{12}-\tilde{X}_2) \\
	H_{i+r_1-c_2}(D^{r_1}\times\tilde{X}_2,D^{r_1}\times\tilde{X}_2-\tilde{X}_1) \ar[u]_-{\cong}
		& H_{i+r_1-c_2}(D^{r_1}\times\tilde{X}_2\cap\tilde{U}_1,D^{r_1}\times\tilde{X}_2\cap\tilde{U}_1-\tilde{X}_1)
			\ar[l]_-{\cong} \ar[r]^-{j_1^\ast(\tilde{u}_1\times1)\frown} \ar[u]_-{\cong}
		& H_{i-c_1-c_2}(D^{r_1}\times\tilde{X}_2\cap\tilde{U}_1). \ar[u]_-{\cong}
}}$
\end{tabular}
\end{turn}
\newpage
\[
\xymatrix{
	H_{i-c_1}(\tilde{X}_1) \ar@{=}[r] \ar[d]^-{g_{r_2}\times} \ar@{}[dr]|{(-1)^{(i-c_1)r_2}}
		& H_{i-c_1}(\tilde{X}_1) \ar[d]^-{\times g_{r_2}} \\
	H_{i-c_1+r_2}((D^{r_2},S^{r_2-1})\times\tilde{X}_1) \ar[r]^-{\cong} \ar[d]
		& H_{i-c_1+r_2}(X_3^I\times(D^{r_2},S^{r_2-1})) \ar[d] \\
	H_{i-c_1+r_2}(D^{r_2}\times\tilde{X}_1,D^{r_2}\times\tilde{X}_1-\tilde{X}_2) \ar[r]^-{\cong}
		& H_{i+r_2}(X_3^I\times D^{r_2},X_3^I\times D^{r_2}-\tilde{X}_2) \\
	H_{i-c_1+r_2}(D^{r_2}\times\tilde{X}_1\cap\tilde{U}_2,D^{r_2}\times\tilde{X}_1\cap\tilde{U}_2-\tilde{X}_2)
			\ar[r]^-{\cong} \ar[u]_-{\cong} \ar[d]^-{\tilde{u}_2\frown} \ar@{}[dr]|{(-1)^{d_3r_2}}
		& H_{i-c_1+r_2}(\tilde{X}_1\times D^{r_2}\cap\tilde{U}_2',\tilde{X}_1\times D^{r_2}\cap\tilde{U}_2'-\tilde{X}_2)
			\ar[u]_-{\cong} \ar[d]^-{\tilde{u}_2'\frown} \\
	H_{i-c_1-c_2}(D^{r_2}\times\tilde{X}_1\cap\tilde{U}_2) \ar[r]^-{\cong}
		& H_{i-c_1-c_2}(\tilde{X}_1\times D^{r_2}\cap\tilde{U}_2') \\
	H_{i-c_1-c_2}(E) \ar[r]^-{\cong} \ar[u]_-{\cong}
		& H_{i-c_1-c_2}(E) \ar[u]_-{\cong}
}
\]
Combining these diagrams, we obtain the desired commutativity of the shriek maps.
\end{proof}

Let us consider a more geometric situation, which will be used to define the local Reidemeister trace.
For the the maps
\[
	X_1\xrightarrow{\delta_1}X_3\xleftarrow{\delta_2}X_2
\]
between smooth connected oriented manifolds, we suppose that $X_2$ is closed and $\delta_2$ is an embedding.
We denote the dimensions by $d_i=\dim X_i$.
To define the homotopy pullback, consider the pullback square
\[
\xymatrix{
	\tilde{X}_1 \ar[d]^-{\epsilon_1}
		& E \ar[l] \ar[d]^-{\pi} \\
	X_3
		& X_2, \ar[l]
}
\]
where $\tilde{X}_1$ is given by
\[
	\tilde{X}_1=\{(x,\ell)\in X_1\times X_3^I\mid\ell(0)=\delta_1(x)\}
\]
and $\epsilon_1(x,\ell)=\ell(1)$.
We denote the canonical inclusion by $i\colon X_1\to \tilde{X}_1$, take a tubular neighborhood $U_2\subset X_3$ of $X_2$ and set $\tilde{U}_2=\epsilon_1^{-1}U_2$.
The Thom classes of $U_2$ and $\tilde{U}_2$ are denoted by $u_2$ and $\tilde{u}_2=\epsilon_1^{\ast}u_2$, respectively.

We take a closed subset $A\subset X$ such that there is a neighborhood $V\subset X_1$ of $A$ such that $\delta_1(V-A)\subset X_3-X_2$.
Then we may suppose that $\delta_1(V-A)\subset U_2-X_2$ replacing by a smaller neighborhood $V$ if necessary.
A local version of the shriek map is given as follows.
Note that the following map is independent of the choice of the neighborhood $V$.

\begin{dfn}
Define the maps
\begin{align*}
	(\delta_2)_A^!&\colon H_i(X_1,X_1-A)\to H_{i-d_2+d_3}(X_2),\\
	(\tilde{\delta}_2)_A^!&\colon H_i(X_1,X_1-A)\to H_{i-d_2+d_3}(E)
\end{align*}
by the composites
\begin{align*}
	(\delta_2)_A^!&\colon H_i(X_1,X_1-A)\xleftarrow{\cong}H_i(V,V-A)
		\xrightarrow{(\delta_1)_\ast}H_i(U_2,U_2-X_2)\xrightarrow{u_2\frown}H_{i-d_2+d_3}(U_2)\xleftarrow{\cong}H_{i-d_2+d_3}(X_2),\\
	(\tilde{\delta}_2)_A^!&\colon H_i(X_1,X_1-A)\xleftarrow{\cong}H_i(V,V-A)
		\xrightarrow{i_\ast}H_i(\tilde{U}_2,\tilde{U}_2-E)\xrightarrow{\tilde{u}_2\frown}H_{i-d_2+d_3}(\tilde{U}_2)\xleftarrow{\cong}H_{i-d_2+d_3}(E).
\end{align*}
\end{dfn}

For a compact subset $A\subset X_1$, the \textit{fundamental class} \cite[Section VIII.4]{MR1335915} $o_A\in H_{d_1}(X_1,X_1-A)$ is characterized by the property that the image of $o_A$ under the map
\[
	H_{d_1}(X_1,X_1-A)\to H_{d_1}(X_1,X_1-\{x\})
\]
is the class representing the orientation for any $x\in A$.
For any compact subset $A$, such a homology class exists uniquely.
For another compact subset $A'\subset A$, $o_A$ is mapped to $o_{A'}$ under the map
\[
	H_{d_1}(X_1,X_1-A)\to H_{d_1}(X_1,X_1-A').
\]
Note that if $X_1$ is closed, then the fundamental class $o_{X_1}$ of $X_1$ is the usual fundamental class $[X_1]$ and the shriek maps $\delta_{X_1},\tilde{\delta}_{X_1}$ coincide with $\delta,\tilde{\delta}$ defined in \ref{dfn_shriek}.

\begin{prp}
\label{prp_shriek_additivity}
Under the above setting, we suppose that $X_1$ is closed and $A$ is decomposed into the disjoint union
\[
	A=A_1\sqcup A_2
\]
of compact subsets $A_1,A_2\subset X_1$.
Then the following equalities hold:
\begin{align*}
	(\delta_2)_A^!=(\delta_2)_{A_1}^!o_{A_1}+(\delta_2)_{A_2}^!o_{A_2},
		\qquad (\tilde{\delta}_2)_A^!o_A=(\tilde{\delta}_2)_{A_1}^!o_{A_1}+(\tilde{\delta}_2)_{A_2}^!o_{A_2}.
\end{align*}
\end{prp}

\begin{proof}
It is sufficient to prove the second equation since the first equation is obtained from the latter applying the map $\pi_\ast$.
We can find neighborhoods $V_j\subset X_1$ of $A_j$ for $j=1,2$ such that $\delta_1(V_j-A_j)\subset X_3-X_2$ and $V_1\cap V_2=\emptyset$.
Then we have the following commutative diagram:
\[
\xymatrix{
	H_{d_1}(X_1,X_1-A) \ar@{=}[r] \ar[d]_-{i_\ast}
		& H_{d_1}(X_1,X_1-(A_1\sqcup A_2)) \ar[d]_-{i_\ast}
		& H_{d_1}(V_1,V_1-A_1)\oplus H_{d_1}(V_2,V_2-A_2) \ar[l]_-{\cong} \ar[d]^-{i_\ast} \\
	H_{d_1}(\tilde{X}_1) \ar[r] \ar[d]_-{(\epsilon_1)_\ast}
		& H_{d_1}(\tilde{X}_1,\tilde{X}_1-E) \ar[d]_-{(\epsilon_1)_\ast}
		& H_{d_1}(\tilde{U},\tilde{U}-E) \ar[l]_-{\cong} \ar[d]^-{(\epsilon_1)_\ast} \\
	H_{d_1}(X_3) \ar[r]
		& H_{d_1}(X_3,X_3-X_2)
		& H_{d_1}(U_2,U_2-X_2). \ar[l]^-{\cong}
}
\]
Let $a_j\in H_{d_1}(V_j,V_j-A_j)$ be the image of $o_{A_j}\in H_{d_1}(X_1,X_1-A_j)$ under the excision isomorphism.
For the Thom class $u_2\in H_{d_3-d_2}(U_2,U_2-X_2)$, we obtain
\begin{align*}
	\tilde{\delta}_2^!i_\ast o_A
		=\tilde{q}_\ast(\tilde{u}_2\frown(i_\ast a_1+i_\ast a_2))
		=(\tilde{\delta}_2)_{A_1}^!o_{A_1}+(\tilde{\delta}_2)_{A_2}^!o_{A_2}.
\end{align*}
where $\tilde{q}\colon\tilde{U}_2\xrightarrow{\simeq}E$ is the projection of the normal bundle.
\end{proof}

\section{Local coincidence Lefschetz and Reidemeister traces}
\label{invariant traces}

In this section, we will introduce the local Reidemeister trace and check the additivity axiom appearing in \cite{MR2529499}.
Combining with the local computation of the coincidence Reidemeiter trace, we will prove that the geometric Reidemeister trace $\rho^{\Geom}(f,g)$ defined in Section \ref{geometric coincidence} is homotopy invariant.

Let $f,g\colon M\to N$ be continuous maps between smooth connected oriented manifolds and $m=\dim M$, $n=\dim N$.
Using the space
\[
	\tilde{M}_{f,g}:=\{(x,\ell_1,\ell_2)\in M\times N^I\times N^I\mid\ell_1(0)=f(x),\ell_2(0)=g(x)\}
\]
and maps
\[
	\epsilon\colon\tilde{M}_{f,g}\to N\times N,\qquad (x,\ell_1,\ell_2)\mapsto(\ell_1(1),\ell_2(1)),
\]
the homotopy equalizer $\Hoeq(f,g)$ appeared in Section \ref{geometric coincidence} is described as
\[
\xymatrix{
	M \ar[r]^-{i} \ar[dr]_-{(f\times g)\circ\Delta}
		& \tilde{M}_{f,g} \ar[d]^-{\epsilon}
		& \Hoeq(f,g) \ar[l]_-{\tilde{\Delta}} \ar[d]^-{\pi} \\
	& N\times N
		& N, \ar[l]^-{\Delta}
}
\]
where the square is pullback, and the map $i\colon M\to\tilde{M}_{f,g}$ is the canonical inclusion given by
\[
	i(x)=(x,(\text{the constant path at }f(x)),(\text{the constant path at }g(x))).
\]
The map $i$ is a homotopy equivalence.
Let $A\subset M$ be a compact subset and a neighborhood $V\subset M$ of $A$ such that $V\cap\Eq(f,g)=A$.
By the construction of the shriek maps in Section \ref{shriek maps} and this pullback square, we obtain the commutative diagram
\[
\xymatrix{
	H_\ast(M,M-A) \ar[r]^-{\tilde{\Delta}_A^!} \ar[dr]_-{\Delta_A^!}
		& H_{\ast-n}(\Hoeq(f,g)) \ar[d]^-{\pi_\ast} \\
		& H_{\ast-n}(N).
}
\]
Desired homotopy invariant constructions are obtained as follows.
\begin{dfn}
We define the maps
\begin{align*}
	R(f,g)_A&:=\tilde{\Delta}_A^!\colon H_\ast(M)\to H_{\ast-n}(\Hoeq(f,g)),\\
	\varLambda(f,g)_A&:=\Delta_A^!\colon H_\ast(M)\to H_{\ast-n}(N).
\end{align*}
We call $R(f,g)$ the \textit{(local) Reidemeister homomorphism} and $\varLambda(f,g)$ the \textit{(local) Lefschetz homomorphism}.
In particular, for the evaluation at the fundamental class $o_A\in H_m(M,M-A)$
\begin{align*}
	\rho(f,g)_A&:=R(f,g)_Ao_A\in H_{m-n}(\Hoeq(f,g)),\\
	\lambda(f,g)_A&:=\varLambda(f,g)_Ao_A\in H_{m-n}(N),
\end{align*}
we call $\rho(f,g)_A$ the \textit{(local) coincidence Reidemeister trace} and $\lambda(f,g)_A$ the \textit{(local) coincidence Lefschetz trace}.
If $M$ is closed, we will denote simply by $R(f,g)_M=R(f,g)$ and so on.
\end{dfn}

Like the case of the Lefschetz trace for fixed points, there is a trace formula for the coincidence Lefschetz trace.
For details, see \cite{MR3463529,MR1853657}.

By the homotopy invariance of the Pontrjagin--Thom construction, all of $R(f,g), \varLambda(f,g), \rho(f,g)$ and $\lambda(f,g)$ depend only on the homotopy classes of $f$ and $g$.
For the map $\pi_\ast\colon H_\ast(\Hoeq(f,g))\to H_\ast(N)$, we have the formulas
\[
	\pi_\ast\circ R(f,g)_A=\varLambda(f,g)_A,\qquad \pi_\ast\rho(f,g)_A=\lambda(f,g)_A.
\] 
In particular, if $m=n$, $R(f,g)_A$ is nontrivial only on $H_m(M)$ and thus can be recovered from $\rho(f,g)_A$.

The following fact is obvious by construction.

\begin{prp}
If $M$ is closed and $f$ and $g$ can be deformed into maps $f',g'$ of which the equalizer $\Eq(f',g')$ is empty, then all of $R(f,g)$, $\varLambda(f,g)$, $\rho(f,g)$ and $\lambda(f,g)$ are trivial.
\end{prp}

The following proposition follows immediately from Proposition \ref{prp_shriek_additivity}.

\begin{prp}
\label{prp_trace_additivity}
For a compact subset $A\subset M$ as above, if $A$ is the disjoint union of compact subsets $A_1,A_2\subset M$, then the following equalities hold:
\[
	\rho(f,g)_A=\rho(f,g)_{A_1}+\rho(f,g)_{A_2},
		\qquad\lambda(f,g)_A=\lambda(f,g)_{A_1}+\lambda(f,g)_{A_2}
\] 
\end{prp}

Now let us prove that the geometric definitions of the coincidence Lefschetz and Reidemeister traces in Section \ref{geometric coincidence} are equivalent to those in this section.
Decompose $\Eq(f,g)$ as the disjoint union of the connected components
\[
	\Eq(f,g)=L_1\sqcup\cdots\sqcup L_k
\]
and take a tubular neighborhood $V_j\subset M$ of $L_j$ for each $j$ such that $V_i\cap V_j=\emptyset$ for $i\ne j$.
We may assume that $L_j$ is orientable if and only if $j\le k_0$ and fix the orientation of $L_j$ for $j\le k_0$.
Let $U\subset N\times N$ be a tubular neighborhood of $\Delta(N)$ and $\tilde{U}=\epsilon^{-1}(U)\subset\tilde M_{f,g}$ its inverse image by the projection $\epsilon$.
We may assume that $\epsilon\circ i(V_j-L_j)\subset U-\Delta(N)$.

\begin{thm}
\label{thm_local_trace}
For $j\le k_0$, the following equality holds:
\[
	\rho(f,g)_{L_j}=\ind(f,g;L_j)\iota_\ast[L_j].
\]
For $j>k_0$, $\rho(f,g)_{L_j}$ vanishes.
\end{thm}

\begin{proof}
Let $a_j\in H_m(V_j,V_j-L_j)$ be the image of $o_{L_j}\in H_m(M,M-L_j)$ under the excision isomorphism.
For the Thom class $u\in H_m(U,U-\Delta(N))$, we have
\begin{align*}
	\rho(f,g)_{L_j}=\tilde{\Delta}_{L_j}^!o_{L_j}
		=\tilde{q}_\ast(\epsilon^\ast u\frown i_\ast a_j)
		=\tilde{q}_\ast i_\ast(i^\ast\epsilon^\ast u\frown a_j)
		=\tilde{q}_\ast i_\ast(\eta_j^\ast u\frown a_j),
\end{align*}
where $\tilde{q}\colon\tilde{U}\xrightarrow{\simeq}\Hoeq(f,g)$ is the projection of the normal bundle and $\eta_j:=(f\times g)\circ\Delta|_{V_j}\colon V_j\to U$.
For $x\in L_j$, take neighborhoods $D_x\subset M$ of $x$ and $D'_{f(x)}\subset N$ of $f(x)$. 
To compute $\eta_j^\ast u$, consider the following commutative diagram:
\[
\xymatrix{
	(D_x,D_x-L_j) \ar[r]^-{\cong} \ar[d]_-{\eta_j}
		& (\mathbb{R}^m,\mathbb{R}^m-\mathbb{R}^{m-n}\times\{0\}) \ar[r]^-{\simeq} \ar[d]
		& (\mathbb{R}^n,\mathbb{R}^n-\{0\}) \ar[d]^-{\Phi_x} \\
	(D'_{f(x)}\times D'_{f(x)},D'_{f(x)}\times D'_{f(x)}-\Delta) \ar[r]_-{\cong}
		& (\mathbb{R}^n\times\mathbb{R}^n,\mathbb{R}^n\times\mathbb{R}^n-\Delta) \ar[r]_-{\simeq}
		& (\mathbb{R}^n,\mathbb{R}^n-\{0\}),
}
\]
where $\Delta$ denotes the appropriate diagonal subsets, the top right arrow is given by the projection on to the latter $n$ factors, and the bottom right arrow by the difference map $(y_1,y_2)\mapsto y_1-y_2$.
This verifies the equality $\eta_j^\ast u=\ind(f,g;L_j)v_j$ where $v_j\in H_n(V_j,V_j-L_j)$ is the Thom class of the normal bundle.
Thus we can continue the above computation as follows:
\begin{align*}
	\rho(f,g)_{L_j}
		&=\tilde{q}_\ast i_\ast(\eta_j^\ast u\frown a_j)\\
		&=\ind(f,g;L_j)\tilde{q}_\ast i_\ast(v_j\frown a_j)\\
		&=\ind(f,g;L_j)\tilde{q}_\ast i_\ast[L_j]
		=\ind(f,g;L_j)\iota_\ast[L_j]		
\end{align*}
by our convention of the orientation of $V_j$ in Section \ref{geometric coincidence}.
The rest of the theorem follows from the fact that $H_m(V_j,V_j-L_j)=0$ for $j>k_0$.
\end{proof}

Combining Proposition \ref{prp_trace_additivity} and Theorem \ref{thm_local_trace}, our Reidemeister traces coincide as follows.

\begin{cor}
\label{cor_geom_equals_inv}
Let $f,g\colon M\to N$ as above.
If $\Eq(f,g)\subset M$ is a smooth closed submanifold of dimension $m-n$, then
\[
	\rho(f,g)=\rho^{\Geom}(f,g),\qquad\lambda(f,g)=\lambda^{\Geom}(f,g).
\]
As a consequence, $\rho^{\Geom}(f,g)$ and $\lambda^{\Geom}(f,g)$ depend only on the homotopy classes of $f$ and $g$.
\end{cor}

The invariant considered in Remark \ref{rem_another_Lefschetz} is also a homotopy invariant by the following corollary.

\begin{cor}
\label{cor_another_Lefschetz}
The image of $\rho(f,g)$ under the map $H_{m-n}(\Hoeq(f,g))\to H_{m-n}(M)$ is
\[
\sum_{j=1}^{k_0}\ind(f,g;L_j)[L_j]\in H_{m-n}(M).
\]
As a consequence, the above homology class is also a homotopy invariant.
\end{cor}

\section{Some properties of coincindece Lefschetz and Reidemeister traces}
\label{properties}

In this section, we study symmetry, product and naturality of coincindece Lefschetz and Reidemeister traces.
For coincidence Lefschetz traces, they have been already proved by Saveliev \cite{MR1853657}.

Let $f,g\colon M\to N$ be continuous maps between smooth connected closed oriented manifolds and $m=\dim M$, $n=\dim N$.
Consider the commutative diagram
\[
\xymatrix{
	\tilde{M}_{f,g} \ar[r]^-{\epsilon} \ar[d]
		& N\times N \ar[d]^-T
		& N \ar[l]_-{\Delta} \ar@{=}[d] \\
	\tilde{M}_{g,f} \ar[r]_-{\epsilon}
		& N\times N
		& N, \ar[l]^-{\Delta}
}
\]
where $\tilde{M}_{f,g}\to\tilde{M}_{g,f}$ reverses the path connecting $f(x)$ and $g(x)$ and $T$ is the switching of the factors.
This diagram induces the homeomorphism on the pullback:
\[
	\tilde{T}\colon\Hoeq(f,g)\to\Hoeq(g,f).
\]

\begin{prp}
\label{prp_symmetry}
The following equalities hold:
\begin{align*}
	&R(g,f)=(-1)^n\tilde{T}_\ast\circ R(f,g),
		&&\varLambda(g,f)=(-1)^n\varLambda(f,g), \\
	&\rho(g,f)=(-1)^n\tilde{T}_\ast\rho(f,g),
		&&\lambda(g,f)=(-1)^n\lambda(f,g).
\end{align*}
\end{prp}

\begin{proof}
We may assume that the tubular neighborhood $U\subset N\times N$ of $\Delta(N)$ satisfies $T(U)\subset U$.
Then it is easy to see that under the isomorphism
\[
	T^\ast\colon H^\ast(U,U-\Delta(N))\to H^\ast(U,U-\Delta(N)),
\]
the Thom class $u$ is mapped to $(-1)^nu$.
We denote the inverse image of $U$ under $\tilde{M}_{f,g}\to N\times N$ by $\tilde{U}\subset\tilde{M}_{f,g}$ and under $\tilde{M}_{g,f}\to N\times N$ by $\tilde{U}'\subset\tilde{M}_{g,f}$.
Then the desired equalities follow from the commutative diagram
\[
\xymatrix{
	H_\ast(M) \ar[r]^-{i_\ast} \ar@{=}[d]
		& H_\ast(\tilde{M}_{f,g}) \ar[r] \ar[d]
		& H_\ast(\tilde{U},\tilde{U}-\Hoeq(f,g)) \ar[r]^-{(-1)^n\pi^\ast u\frown} \ar[d]
		& H_\ast(\tilde{U}) \ar[r]^-{\cong} \ar[d]
		& H_\ast(\Hoeq(f,g)) \ar[d]^-{\tilde{T}_\ast} \\
	H_\ast(M) \ar[r]^-{i_\ast}
		& H_\ast(\tilde{M}_{g,f}) \ar[r]
		& H_\ast(\tilde{U}',\tilde{U}'-\Hoeq(g,f)) \ar[r]^-{\pi^\ast u\frown}
		& H_\ast(\tilde{U}') \ar[r]^-{\cong}
		& H_\ast(\Hoeq(g,f)). \\
}
\]
\end{proof}

Let $f',g'\colon M'\to N'$ be continuous maps between smooth connected closed oriented manifolds and $m'=\dim M'$, $n'=\dim N'$.
Consider the commutative diagram
\[
\xymatrix{
	\tilde{M}_{f,g}\times\tilde{M}'_{f',g'} \ar[r]^-{\epsilon\times\epsilon} \ar[d]
		& (N\times N)\times(N'\times N') \ar[d]^-{S}
		& N\times N' \ar[l]_-{\Delta\times\Delta} \ar@{=}[d] \\
	(M\times M')^{\sim}_{f\times f',g\times g'} \ar[r]_-{\epsilon}
		& (N\times N')\times(N\times N')
		& N\times N', \ar[l]^-{\Delta}
}
\]
where the left vertical arrow is given by
\[
	[(x,\ell_1,\ell_2),(x',\ell_1',\ell_2')]\mapsto(x,x',\ell_1,\ell_1',\ell_2,\ell_2')
\]
and $S$ is the switching of the factors.
This diagram induces the homeomorphism on the pullback:
\[
	\tilde{S}\colon\Hoeq(f,g)\times\Hoeq(f',g')\to\Hoeq(f\times f',g\times g').
\]

\begin{prp}
\label{prp_product}
For $a\in H_\ast(M)$ and $a'\in H_\ast(M')$, the following equalities hold:
\begin{align*}
	R(f\times f',g\times g')(a\times a')&=(-1)^{n'(n+|a|)}\tilde{S}_\ast(R(f,g)a\times R(f',g')a'),\\
	\varLambda(f\times f',g\times g')(a\times a')&=(-1)^{n'(n+|a|)}\varLambda(f,g)a\times\varLambda(f',g')a'.
\end{align*}
In particular, for the coincidence Lefschetz and Reidemeister traces, the following equalities hold:
\begin{align*}
	\rho(f\times f',g\times g')&=(-1)^{n'(n+|a|)}\tilde{S}_\ast(\rho(f,g)\times\rho(f',g')),\\
	\lambda(f\times f',g\times g')&=(-1)^{n'(n+|a|)}\lambda(f,g)\times\lambda(f',g').
\end{align*}
\end{prp}

\begin{proof}
Let $U\subset N\times N$ be a tubular neighborhood of $\Delta(N)$ and $U'\subset N'\times N'$ be a tubular neighborhood of $\Delta(N')$.
We denote the corresponding Thom classes by $u\in H_n(U,U-\Delta(N))$ and $u'\in H_{n'}(U',U'-\Delta(N'))$, respectively.
Then $U\times U'\subset(N\times N)\times(N'\times N')$ is a tubular neighborhood of $\Delta(N)\times\Delta(N')$, of which the Thom class is $u\times u'$.
Note that, for the map
\[
	(S^\ast)^{-1}\colon H^\ast(U\times U',U\times U'-(\Delta(N)\times\Delta(N')))\to H^\ast(S(U\times U'),S(U\times U')-(\Delta(N\times N'))),
\]
$(-1)^{nn'}(S^\ast)^{-1}(u\times u')$ is the Thom class of the tubular neighborhood $S(U\times U')\subset (N\times N')^{\times 2}$ of $\Delta(N\times N')$.
By the commutative diagram as in the proof of Proposition \ref{prp_symmetry} and the formula
\[
(u\times u')\frown(\alpha\times\alpha')=(-1)^{n'|\alpha|}(u\frown\alpha)\times(u'\frown\alpha'),
\]
we obtain the desired formulas.
\end{proof}

Let $P$ be a smooth connected closed oriented manifold, and $h\colon P\to M$.
Then we have the following commutative diagram:
\[
\xymatrix{
	P \ar[r]^-{i} \ar[d]_-{h}
		& \tilde{P}_{f\circ h,g\circ h} \ar[d]_-{\tilde{h}}
		& \Hoeq(f\circ h,g\circ h) \ar[l] \ar[d]^-{\tilde{h}} \\
	M \ar[r]^-{i} \ar[dr]_-{(f\times g)\circ\Delta}
		& \tilde{M}_{f,g} \ar[d]
		& \Hoeq(f,g) \ar[l] \ar[d] \\
	& N\times N
		& N, \ar[l]^-{\Delta} 
}
\]
where $\tilde{h}(x,\ell_1,\ell_2)=(h(x),\ell_1,\ell_2)$ and the right squares are pullback.

\begin{prp}
\label{prp_naturality}
The following equalities hold:
\[
	\tilde{h}_\ast\circ R(f\circ h,g\circ h)=R(f,g)\circ h_\ast,
		\qquad \varLambda(f\circ h,g\circ h)=\varLambda(f,g)\circ h_\ast.
\]
In particular, if $\dim P=\dim M=m$, the following equalities hold as well:
\[
	\tilde{h}_\ast\rho(f\circ h,g\circ h)=(\deg h)\rho(f,g),
		\qquad \lambda(f\circ h,g\circ h)=(\deg h)\lambda(f,g),
\]
where $\deg h\in\mathbb{Z}$ is defined by $h_\ast[P]=(\deg h)[M]$.
\end{prp}

\begin{proof}
These equalities follow from the commutative diagram
\[
\xymatrix{
	H_\ast(P) \ar[r] \ar[d]_-{h_\ast}
		& H_\ast(\tilde{h}^{-1}\tilde{U},\tilde{h}^{-1}\tilde{U}-\Hoeq(f\circ h,g\circ h)) \ar[r]^-{\tilde{h}^\ast\epsilon^\ast u\frown} \ar[d]
		& H_\ast(\tilde{h}^{-1}\tilde{U}) \ar[r]^-{\cong} \ar[d]
		& H_\ast(\Hoeq(f\circ h,g\circ h)) \ar[d]^-{\tilde{h}_\ast} \\
	H_\ast(M) \ar[r]
		& H_\ast(\tilde{U},\tilde{U}-\Hoeq(g,f)) \ar[r]^-{\epsilon^\ast u\frown}
		& H_\ast(\tilde{U}) \ar[r]^-{\cong}
		& H_\ast(\Hoeq(f,g)). \\
}
\]
\end{proof}

\begin{rem}
Recall \cite[Corollary 6.6]{MR1853657}.
By Proposition \ref{prp_naturality}, if $a\in H_i(M)$ and $i>n$, we have
\[
	\varLambda(f,f)a=\varLambda(\id,\id)(f_\ast a)=0.
\]
But $R(f,f)a\ne0$ in general.
In Section \ref{example1}, we study such examples.
\end{rem}

We have another pullback square describing $\Hoeq(f,g)$:
\[
\xymatrix{
	N^I \ar[d]_-{(\epsilon_0,\epsilon_1)}
		& \Hoeq(f,g) \ar[l]_-{F} \ar[d] \\
	N\times N
		& M, \ar[l]_-{(f,g)}
}
\]
where $\epsilon_i\colon N^I\to N$ is the evaluation at $i\in I$.
We denote the inclusion to the constant paths by $i\colon N\to N^I$.
One might think this diagram also defines a Reidemeister trace type invariant.
Actually, it essentially coincides with $\rho(f,g)$ as follows.

\begin{thm}
\label{thm_transpose}
For the shriek map $F^!\circ i_\ast\colon H_n(N)\to H_{m-n}(\Hoeq(f,g))$, the following equality holds:
\[
	F^!i_\ast[N]=(-1)^{mn}\rho(f,g).
\]
\end{thm}

\begin{proof}
This theorem immediately follows from Propositions \ref{prp_shriek_property1} and \ref{prp_pullback_shriek}.
\end{proof}

Suppose that $H_\ast(N)$ is a free $R$-module and $b_1,\ldots,b_r$ are a homogeneous basis of $H^\ast(N)$.
Define the dual basis $b^i$ by
\[
	\langle b^i\smile b_j,[N]\rangle=
	\begin{cases}
		1	& i=j \\
		0	& i\ne j.
	\end{cases}
\]
Then, as in \cite[Exercise 8.21 in Section VIII]{MR1335915}, the Poincar\'{e} dual of $\Delta_\ast[N]$ is
\[
	\sum_{i=1}^r(-1)^{|b_i|}b^i\times b_i\in H^n(N\times N).
\]
Using these bases, the invariant considered in Remark \ref{rem_another_Lefschetz} is explicitly computed as follows. 

\begin{cor}
\label{cor_primary_obstruction}
Under the above assumption, the Poincar\'{e} dual of the image of $\rho(f,g)\in H_{m-n}(\Hoeq(f,g))$ under the homomorphism induced from $\Hoeq(f,g)\to M$ is the class
\[
	\sum_{i=1}^r(-1)^{|b_i|}f^\ast(b^i)\smile g^\ast(b_i)\in H^n(M).
\]
In particular, if $f=g$, then this class is equal to $\chi(N)f^\ast(u)$, where $\chi(N)$ is the Euler characteristic of $N$ and $u\in H^n(N)$ is the cohomology class characterized by $\langle u,[N]\rangle=1$.
\end{cor}

\begin{proof}
This follows from Theorem \ref{thm_transpose} and the following commutative diagrams:
\[
\xymatrix{
	H_n(N) \ar[r] \ar[d]
		& H_{m-n}(\Hoeq(f,g)) \ar[d] \\
	H_n(N\times N) \ar[r]
		& H_{m-n}(M),
}
\xymatrix{
	H_n(N\times N) \ar[rr]
		& & H_{m-n}(M) \\
	H^n(N\times N) \ar[rr]^-{(-1)^{mn}f^\ast\smile g^\ast} \ar[u]^-{\frown[N]\times[N]}
		& & H^{m-n}(M). \ar[u]_-{\frown[M]}
}
\]
\end{proof}

\begin{rem}
For the self coincidence of $f$, the class $\chi(N)f^\ast(u)$ is nothing but the primary obstruction in \cite[Proposition 2.11]{MR2147735}.
\end{rem}


\section{Reidemeister map on Thom spectra}
\label{Thom spectra}

In this section, we study the Reidemeister map on Thom spectra and relation with the works by Koschorke \cite{MR2270573} and Ponto \cite[Section 4]{MR3463529}.

Let us recall the Thom spectra associated to vector bundles.
For a vector bundle $\xi$ on $X$, the \textit{Thom spectrum} is the suspension spectrum $\Sigma^\infty(D(\xi)/S(\xi))$ of the quotient space $D(\xi)/S(\xi)$, where $D(\xi)$ and $S(\xi)$ are the disk bundle and the sphere bundle of $\xi$, respectively.
We denote the Thom spectrum of $\xi$ by $X^{\xi}$.
Note that taking the direct product $\mathbb{R}^r\times\xi$ (i.e. the Whitney sum with the trivial bundle) corresponds to the suspension:
\[
	X^{\mathbb{R}^r\times\xi}\cong\Sigma^rX^\xi.
\]
From this, for a virtual vector bundle $\xi$ over $X$ such that the Whitney sum of $\xi$ with the trivial bundle of sufficiently large rank is a genuine vector bundle, the Thom spectrum is also defined as
\[
	X^\xi:=\Sigma^{-r}X^{\mathbb{R}^r\times\xi}
\]
for sufficiently large $r$.
In particular, if $X$ is a closed manifold, the Thom spectrum $X^{-TX}$ of the stable normal bundle $-TX$ is known to be the Spanier--Whitehead dual of $\Sigma^\infty(X_+)$ and is a ring spectrum.
For this, we only use the fact that $X^{-TX}$ admits the unit map $\mathbb{S}\to X^{-TX}$ from the sphere spectrum $\mathbb{S}=\Sigma^\infty S^0$ of which the Hurewicz image is the fundamental class under the Thom isomorphism of $-TX$ if $X$ is orientable (in the coefficient $\mathbb{Z}$).

Most of the arguments we have done can be extended to related Thom spectra.
For example, consider the pullback square as in Section \ref{shriek maps}:
\[
\xymatrix{
	E_1 \ar[r]^{\tilde{\delta}} \ar[d]_-{\pi_1}
		& E_2 \ar[d]^-{\pi_2} \\
	X_1 \ar[r]_-{\delta}
		& X_2,
}
\]
where $\pi_1$ and $\pi_2$ are fibrations.
Taking an embedding $X_1\to D^r$, we have the embeddings $\delta'\colon X_1\to D^{r}\times X_2$ and $\tilde{\delta}'\colon E_1\to D^{r}\times E_2$, and tubular neighborhoods $U\subset D^{r}\times X_2$ of $X_1$ and $\tilde{U}\subset D^{r}\times E_2$ of $E_1$.
Let $\xi$ be a vector bundle on $X_2$ and $\nu$ a normal bundle of $X_1\subset D^r\times X_2$.
Then we obtain the \textit{Pontrjagin--Thom maps}:
\begin{align*}
	\delta^!\colon X_2^\xi\to X_1^{\xi+\nu},
		\qquad\tilde{\delta}^!\colon E_2^{\pi_2^\ast\xi}\to E_1^{\pi_2^\ast(\xi+\nu)}
\end{align*}
by the following commutative diagrams:
\begin{align*}
\xymatrix{
	D^r\times D(\xi) \ar[d] \ar[dr]
		& \\
	\Sigma^r(D(\xi)/S(\xi)) \ar[r]
		& (D^r\times D(\xi))/((D^r\times D(\xi)-q^{-1}(U))\cup D^r\times S(\xi)),\\
	D^r\times D(\pi_2^\ast\xi) \ar[d] \ar[dr]
		& \\
	\Sigma^r(D(\pi_2^\ast\xi)/S(\pi_2^\ast\xi)) \ar[r]
		& (D^r\times D(\pi_2^\ast\xi))/((D^r\times D(\pi_2^\ast\xi)-q^{-1}(\tilde{U}))\cup D^r\times S(\pi_2^\ast\xi)),
}
\end{align*}
where $q$ denotes the projections $D^r\times D(\xi)\to D^r\times X_2$ and $D^r\times D(\pi_2^\ast\xi)\to D^r\times E_2$.
This construction also extended to the case when $\xi$ is a virtual bundle.

Let $f,g\colon M\to N$ be continuous maps between smooth connected closed manifolds and $m=\dim M$, $n=\dim N$.
Consider the diagram
\[
\xymatrix{
	M \ar[r]^-{i} \ar[dr]_-{(f\times g)\circ\Delta}
		& \tilde{M}_{f,g} \ar[d]^-{\epsilon}
		& \Hoeq(f,g) \ar[l]_-{\tilde{\Delta}} \ar[d]^-{\pi} \\
	& N\times N
		& N, \ar[l]^-{\Delta}
}
\]
as in Section \ref{invariant traces}.

\begin{dfn}
The \textit{Reidemeister map}
\[
	R(f,g)\colon M^{-TM}\to\Hoeq(f,g)^{TN-TM}
\]
is defined to be the composite
\[
	M^{-TM}\xrightarrow{i_\ast}\tilde{M}_{f,g}^{-p^\ast TM}\xrightarrow{\tilde{\Delta}^!}\Hoeq(f,g)^{\pi^\ast TN-\tilde{\Delta}^\ast p^\ast TM}=:\Hoeq(f,g)^{TN-TM}
\]
where $p\colon\tilde{M}_{f,g}\to M$ is the canonical projection.
The composite with the unit map
\[
	\mathbb{S}\to M^{-TM}\xrightarrow{R(f,g)}\Hoeq(f,g)^{TN-TM}
\]
defines an element $\rho^\pi(f,g)\in\pi_0(\Hoeq(f,g)^{TN-TM})$ of the stable homotopy group.
We call it the \textit{(homotopical) Reidemeister trace}.
\end{dfn}

By construction, the Reidemeister map $R(f,g)$ coincides with the map $M_+\to\Hoeq(f,g)^{TN}$ given by Ponto \cite[Section 4]{MR3463529} after taking the Thom spectrum of $-TM$.
Moreover, the induced homomorphism
\[
H_i(M)\cong H_{i-m}(M^{-TM})\xrightarrow{R(f,g)_\ast}H_{i-m}(\Hoeq(f,g)^{TN-TM})\cong H_{i-n}(\Hoeq(f,g))
\]
coincides with the Reidemeister trace defined in Section \ref{invariant traces}.

The following proposition immediately follows from the property of the unit map $\mathbb{S}\to M^{-TM}$.

\begin{prp}
If $M$ and $N$ are oriented, the Hurewicz image of the homotopical Reidemeister trace $\rho^\pi(f,g)$ is just the ``homological'' Reidemeister trace $\rho(f,g)\in H_{m-n}(\Hoeq(f,g))$ under the appropriate Thom isomorphisms.
\end{prp}

Koschorke \cite{MR2270573} defined the homotopy invariant $\tilde{\omega}(f,g)$ as an element of the bordism group $\Omega_{m-n}(\Hoeq(f,g);TN-TM)$ introduced by Hatcher--Quinn \cite{MR0353322}.
Suppose that $\Eq(f,g)\subset M$ is an $(m-n)$-dimensional closed submanifold and decomposed as the disjoint union of path components:
\[
	\Eq(f,g)=L_1\sqcup\cdots\sqcup L_k.
\]
Under the isomorphism $\Omega_{m-n}(\Hoeq(f,g);TN-TM)\cong\pi_0(\Hoeq(f,g)^{TN-TM})$, Koschorke's invariant $\tilde{\omega}(f,g)$ corresponds to the sum of the composites
\[
	\mathbb{S}\to L_j^{-TL_j}\xrightarrow{\iota_\ast}\Hoeq(f,g)^{TN-TM}
\]
for $j=1,\ldots,k$, where $\iota_\ast$ is the map induced from the inclusion $\iota\colon\Eq(f,g)\to\Hoeq(f,g)$ and the canonical fiberwise stable map $S(-TL_j)\to S(\iota^\ast(TN-TM))$ over $L_j$.

\begin{rem}
Koschorke assumed that $(f,g)\colon M\to N\times N$ is transverse to the diagonal subset.
Under this assumption, the coincidence index of each orientable $L_j$ must be $\pm1$.
Our assumption is slightly weaker than Koschorke's but the definition of $\tilde{\omega}(f,g)$ extends as above.
\end{rem}

\begin{thm}
\[
	\tilde{\omega}(f,g)=\rho^\pi(f,g).
\]
\end{thm}

\begin{proof}
Considering the appropriate tubular neighborhood of $\Eq(f,g)$, one can see that there is a homotopy commutative diagram of spectra
\[
\xymatrix{
	\mathbb{S} \ar[r] \ar[dr]
		& M^{-TM} \ar[d] \ar[dr]^-{R(f,g)}
		& \\
		& \bigvee_{j=1}^kL_j^{-TL_j} \ar[r]
		& \Hoeq(f,g)^{TN-TM},
}
\]
where each $M^{-TM}\to L_j^{-TL_j}$ is induced from the inclusion $L_j\to M$.
Thus we obtain the equality $\tilde{\omega}(f,g)=\rho^\pi(f,g)$.
\end{proof}

If the codimension between $M$ and $N$ is small, it is known that the homotopical Reidemeister trace is a very strong invariant.
The following is proved by Koschorke \cite[Theorem 1.10]{MR2270573}

\begin{thm}[Koschorke]
Suppose $m<2n-2$.
Then $f$ and $g$ can be deformed into the coincidence-free maps if and only if $\rho^\pi(f,g)$ vanishes.
\end{thm}

\section{Various Nielsen numbers}
\label{Nielsen numbers}

In this section we generalize Nielsen numbers.
Some of them have already been studied by Koschorke \cite{MR2991953}.

Let $f,g\colon M\to N$ be continuous maps between smooth connected closed manifolds and $m=\dim M$, $n=\dim N$.
For an element $\alpha\in\pi_0(\Hoeq(f,g))$, we denote the corresponding path component by $\Hoeq(f,g)_\alpha$.
Under this notation, we have the following obvious decomposition:
\[
	\Hoeq(f,g)^{TN-TM}\simeq\bigvee_{\alpha\in\pi_0(\Hoeq(f,g))}\Hoeq(f,g)_\alpha^{TN-TM}
\]
For a homology theory $h_\ast$ given by a ring spectrum, we denote the homology class given by the unit map by $1\in h_0(\mathbb{S})$.

\begin{dfn}
The $h$-\textit{Reidemeister trace} $\rho^h(f,g)$ and $h$-\textit{Lefschetz trace} $\lambda^h(f,g)$ are defined as follows:
\[
	\rho^h(f,g)=R(f,g)_\ast1\in h_0(\Hoeq(f,g)^{TN-TM}),\qquad \lambda^h(f,g)=\pi_\ast\rho^h(f,g)\in h_0(N^{-TN}),
\]
where $R(f,g)$ is the Reidemeister map and $\pi\colon\Hoeq(f,g)\to N$ is the canonical projection.
\end{dfn}

\begin{dfn}
We denote the number of elements $\alpha\in\pi_0(f,g)$ such that the image of $\rho^h(f,g)$ under the projection onto $h_0(\Hoeq(f,g)_\alpha^{TN-TM})$ is nontrivial by $\tilde{N}^h(f,g)$.
Similarly, we denote the number of elements $\alpha\in\pi_0(f,g)$ such that the image of $\rho^h(f,g)$ under the projection onto $h_0(\Hoeq(f,g)_\alpha^{TN-TM})$ has nontrivial image in $h_0(M^{f^\ast TN-TM})$ by $N^h(f,g)$.
We call these numbers by \textit{Nielsen numbers}.
\end{dfn}

In particular, we will denote the Reidemeister trace in the integral homology by $\rho^{\mathbb{Z}}(f,g)$.
The following lemma is easy to verify.

\begin{lem}
\begin{enumerate}
\item
The following inequalities hold:
\[
	\tilde{N}^h(f,g)\ge N^h(f,g),\qquad \tilde{N}^\pi(f,g)\ge\tilde{N}^h(f,g), \qquad N^\pi(f,g)\ge N^h(f,g).
\]
\item
If $f(x)\ne g(x)$ for any $x\in M$, then $\tilde{N}^h(f,g)=N^h(f,g)=0$.
\item
If $M$ and $N$ are orientable in the integral homology and $m=n$, then $\tilde{N}^\pi(f,g)=\tilde{N}^{\mathbb{Z}}(f,g)$.
\end{enumerate}
\end{lem}

As an example, let us consider the maps between spheres.
For $h=\pi$, see \cite{MR2270573}.

\begin{ex}
Let $m>n\ge2$ and take maps $f,g\colon S^m\to S^n$.
For the fibration
\[
\Omega S^n\to\Hoeq(f,g)\to S^m,
\]
the fiber inclusion $\Omega S^n\to\Hoeq(f,g)$ is $(m-2)$-connected.
Since $TS^m$ and $TS^n$ are stably trivial, we have the isomorphisms
\[
\pi_0(\Hoeq(f,g)^{TS^n-TS^m})
	\cong\pi_{m-n}(\Sigma^\infty\Hoeq(f,g))
	\cong\pi_{m-n}(\Sigma^\infty\Omega S^n)
	\cong\bigoplus_{k\ge1}\pi_{m-1-k(n-1)}(\mathbb{S}).
\]
Koschorke \cite[Theorem 1.14]{MR2270573} proved that, under the composite of these isomorphisms, the element $\rho^\pi(f,g)=\tilde{\omega}(f,g)$ corresponds to
\[
\varGamma(f)-(-1)^{k(n-1)}\varGamma(g),
\]
where $\varGamma\colon\pi_m(S^n)\to\oplus_{k\ge1}\pi_{m-1-k(n-1)}(\mathbb{S})$ is the stabilized James--Hopf invariant.
Note that all the nontrivial homomorphisms on homology induced from the James--Hopf invariant are just the degree and the usual Hopf invariant.
If $n$ is odd or $m\ne2n-1$, $\rho^{\mathbb{Z}}(f,g)$ is trivial and hence $\tilde{N}^{\mathbb{Z}}(f,g)=N^{\mathbb{Z}}(f,g)=0$.
If $n$ is even and $m=2n-1$, $\rho^{\mathbb{Z}}(f,g)\in H_{n-1}(\Hoeq(f,g))\cong H_{n-1}(\Omega S^n)\cong\mathbb{Z}$ corresponds to the integer
\[
H(f)-H(g),
\]
where $H(f)$ denotes the Hopf invariant of $f$.
If $H(f)\ne H(g)$, then $\tilde{N}^{\mathbb{Z}}(f,g)=1$.
But $N^{\mathbb{Z}}(f,g)$ is trivial for any $f,g\colon S^{2n-1}\to S^n$ by dimensional reason.
\end{ex}


\section{Shriek map on Serre spectral sequence}
\label{shriek map on ss}

Cohen--Jones--Yan \cite{MR2039760} constructed the shriek map on Serre spectral sequence.
Though they are concentrated on the case for free loop spaces, their construction may be done for more general case.
In this section, we recall their construction.

Consider the following pullback square as in Section \ref{shriek maps}:
\[
\xymatrix{
	E_1 \ar[r]^{\tilde{\delta}} \ar[d]_-{\pi_1}
		& E_2 \ar[d]^-{\pi_2} \\
	X_1 \ar[r]_-{\delta}
		& X_2.
}
\]
We take tubular neighborhoods $U\subset X_2$ of $\delta(X_1)$ and $\tilde{U}=\pi_2^{-1}U\subset E_2$ of $\tilde{\delta}(E_1)$.
We do not need to assume that the homology local system associated to $E_2\to X_2$ is trivial.
We denote the codimension by $d=d_2-d_1$ and the fiber of $\pi_1$ and $\pi_2$ by $F$.
Then, there are increasing filtrations of singular chain complexes
\[
	\{F_pC_\ast(E_1)\}_p,
	\quad \{F_pC_\ast(E_2)\}_p,
	\quad \{F_pC_\ast(\tilde{U})\}_p,
	\quad \{F_pC_\ast(E_2,E_2-\tilde{\delta}(E_1))\}_p,
	\quad \{F_pC_\ast(\tilde{U},\tilde{U}-\tilde{\delta}(E_1))\}_p
\]
and the following associated homology Serre spectral sequences:
\begin{align*}
	\{E^r(E_1),d^r\}\colon&
		&& E^2_{p,q}=H_p(X_1;\mathcal{H}_q(F))&\Longrightarrow&H_{p+q}(E_1),\\
	\{E^r(E_2),d^r\}\colon&
		&& E^2_{p,q}=H_p(X_2;\mathcal{H}_q(F))&\Longrightarrow&H_{p+q}(E_2),\\
	\{E^r(\tilde{U}),d^r\}\colon&
		&& E^2_{p,q}=H_p(U;\mathcal{H}_q(F))&\Longrightarrow&H_{p+q}(\tilde{U}),\\
	\{E^r(E_2,E_2-\tilde{\delta}(E_1)),d^r\}\colon&
		&& E^2_{p,q}=H_p(X_2,X_2-\delta(X_1);\mathcal{H}_q(F))&\Longrightarrow&H_{p+q}(E_2,E_2-\tilde{\delta}(E_1)),\\
	\{E^r(\tilde{U},\tilde{U}-\tilde{\delta}(E_1)),d^r\}\colon&
		&& E^2_{p,q}=H_p(U,U-\delta(X_1);\mathcal{H}_q(F))&\Longrightarrow&H_{p+q}(\tilde{U},\tilde{U}-\tilde{\delta}(E_1)),
\end{align*}
where we denote the local system associated to the fibrations $E_1\to X_1$ or $E_2\to X_2$ by $\mathcal{H}_\ast(F)$.
We omit the precise definition of the filtrations.
For details, see \cite[Section 2]{MR2039760}.
Let $u\in H^d(U,U-\delta(X_1))$ be the Thom class.
We also denote the image of $u$ under the natural homomorphism
\[
	H^d(U,U-\delta(X_1))\to H^d(U,U-\delta(X_1);\mathcal{H}^0(F))
\]
by $u$.
Then the cap product $\pi_2^\ast u\frown$ respects the filtration:
\[
	\pi_2^\ast u\frown\colon F_pC_\ast(\tilde{U},\tilde{U}-\tilde{\delta}(E_1))\to F_{p-d}(\tilde{U}).
\]
Thus it induces the map of spectral sequences
\[
	\varphi\colon E^r_{p,q}(\tilde{U},\tilde{U}-\tilde{\delta}(E_1))\to E^r_{p-d,q}(\tilde{U})
\]
of bidegree $(-d,0)$ in the sense that
\begin{enumerate}
\item
for $x\in E^r_{p,q}(\tilde{U},\tilde{U}-\tilde{\delta}(E_1))$, $\varphi(d^rx)=(-1)^dd^r\varphi(x)$,
\item
the following diagrams commute:
\[
\xymatrix{
	H_p(U,U-\delta(X_1);\mathcal{H}_q(F)) \ar[r]^-{u\frown} \ar[d]_-{\cong}
		& H_{p-d}(U;\mathcal{H}_q(F)) \ar[d]^-{\cong} \\
	E^2_{p,q}(\tilde{U},\tilde{U}-\tilde{\delta}(E_1)) \ar[r]_-{\varphi}
		& E^2_{p-d,q}(\tilde{U}),
}
\]
\[
\xymatrix{
	\ker d^r_{p,q}/\im d^r_{p+r,q-r+1} \ar[r]^-{\varphi} \ar[d]_-{\cong}
		& \ker d^r_{p-d,q}/\im d^r_{p+r-d,q-r+1} \ar[d]^-{\cong} \\
	E^{r+1}_{p,q}(\tilde{U},\tilde{U}-\tilde{\delta}(E_1)) \ar[r]_-{\varphi}
		& E^{r+1}_{p-d,q}(\tilde{U}),
}
\]
\[
\xymatrix{
	E^\infty_{p,q}(\tilde{U},\tilde{U}-\tilde{\delta}(E_1)) \ar[r]^-{\varphi} \ar[d]_-{\cong}
		& E^\infty_{p-d,q}(\tilde{U}) \ar[d]^-{\cong} \\
	F_pH_{p+q}(\tilde{U},\tilde{U}-\tilde{\delta}(E_1))/F_{p-1}H_{p+q}(\tilde{U},\tilde{U}-\tilde{\delta}(E_1)) \ar[r]_-{\pi_2^\ast u\frown}
		& F_{p-d}H_{p+q-d}(\tilde{U})/F_{p-d-1}H_{p+q-d}(\tilde{U}).
}
\]
\end{enumerate}

Therefore, we obtain the following map.

\begin{thm}
Under the above setting, there is a map of bidegree $(-d,0)$ between spectral sequences
\[
	\tilde{\delta}^!\colon E^r_{p,q}(E_2)\to E^r_{p-d,q}(E_1)
\]
satisfying the following properties:
\begin{enumerate}
\item
for $x\in E^r_{p,q}(E_2)$, $\tilde{\delta}^!(d^rx)=(-1)^dd^r\tilde{\delta}^!(x)$,
\item
the following diagrams commute:
\[
\xymatrix{
	H_p(X_2;\mathcal{H}_q(F)) \ar[r]^-{\delta^!} \ar[d]_-{\cong}
		& H_{p-d}(X_1;\mathcal{H}_q(F)) \ar[d]^-{\cong} \\
	E^2_{p,q}(E_2) \ar[r]_-{\tilde{\delta}^!}
		& E^2_{p-d,q}(E_1), \\
	\ker d^r_{p,q}/\im d^r_{p+r,q-r+1} \ar[r]^-{\tilde{\delta}^!} \ar[d]_-{\cong}
		& \ker d^r_{p-d,q}/\im d^r_{p+r-d,q-r+1} \ar[d]^-{\cong} \\
	E^{r+1}_{p,q}(E_2) \ar[r]_-{\tilde{\delta}^!}
		& E^{r+1}_{p-d,q}(E_1), \\
	E^\infty_{p,q}(E_2) \ar[r]^-{\tilde{\delta}^!} \ar[d]_-{\cong}
		& E^\infty_{p-d,q}(E_1) \ar[d]^-{\cong} \\
	F_pH_{p+q}(E_2)/F_{p-1}H_{p+q}(E_2) \ar[r]_-{\tilde{\delta}^!}
		& F_{p-d}H_{p+q-d}(E_1)/F_{p-d-1}H_{p+q-d}(E_1).
}
\]
\end{enumerate}
\end{thm}
\begin{proof}
The map $\tilde{\delta}^!$ is defined by the composite
\[
	E^r_{p,q}(E_2)
		\to E^r_{p,q}(E_2,E_2-\tilde{\delta}(E_1))
		\cong E^r_{p,q}(\tilde{U},\tilde{U}-\tilde{\delta}(E_1))
		\xrightarrow{\varphi} E^r_{p,q}(\tilde{U})
		\cong E^r_{p,q}(E_1).
\]
Then the properties (1) and (2) are obvious by definition and the properties of $\varphi$.
\end{proof}

\section{Self-coincidence of $S^1$-bundles}
\label{example1}

The self-coincidence problems of the projections of $S^1$-bundles on $\mathbb{C}P^n$ are studied by Dold--Gon\c{c}alves \cite{MR2147735}.
Since the difference of the dimensions of the source and the target is $1$ and $\mathbb{C}P^2$ is simply connected, the coincidence Lefschetz trace does not work.
In this section, using the shriek map between the Serre spectral sequences recalled in Section \ref{shriek map on ss}, we investigate what can we say about our generalized Reidemeister trace.
The coefficient ring is $\mathbb{Z}$ in this section.

Let $x\in H^2(\mathbb{C}P^n)$ and $[\mathbb{C}P^n]$ be the generators satisfying the equality
\[
	\langle x^n,[\mathbb{C}P^n]\rangle=1.
\]
We follow the notation by Dold--Gon\c{c}alves, namely,
\[
	p^k_n\colon E^k_n\to\mathbb{C}P^n
\]
is the principal $S^1$-bundle classified by $kx\in H^2(\mathbb{C}P^n)\cong [\mathbb{C}P^n,\mathbb{C}P^\infty]$.

Let us compute the Reidemeister trace using the Serre spectral sequence.
We denote the homotopy fiber of the maps $E^k_n\to\mathbb{C}P^n\times\mathbb{C}P^n$ and $\Hoeq(p^k_n,p^k_n)\to\mathbb{C}P^n$ by $F$, which is path connected.
Considering the commutative diagram
\[
\xymatrix{
	E^k_n \ar@{=}[r] \ar[d]_-{p^k_n}
		& E^k_n \ar[r]^-{p^k_n} \ar[d]_-{(p^k_n,p^k_n)}
		& \mathbb{C}P^n \ar[d]^-{\Delta} \\
	\mathbb{C}P^n \ar[r]_-{\Delta}
		& \mathbb{C}P^n\times\mathbb{C}P^n \ar@{=}[r]
		& \mathbb{C}P^n\times\mathbb{C}P^n,
}
\]
one can compute $H_1(F)$ and the $E^2$-term of the homology Serre spectral of $E^k_n\to\mathbb{C}P^n\times\mathbb{C}P^n$ as follows:
\begin{align*}
H_1(F)&=\mathbb{Z}\{a,b\},\\
E^2_{p,q}&\cong
	\begin{cases}
		H^{4n-p}(\mathbb{C}P^n\times\mathbb{C}P^n)								& q=0 \\
		H^{4n-p}(\mathbb{C}P^n\times\mathbb{C}P^n)\otimes\mathbb{Z}\{a,b\}		& q=1
	\end{cases}
\end{align*}
such that $d^2((x^{n-1}\times x^n)\otimes1)=(x^n\times x^n)\otimes(ka+b)$ and $d^2((x^n\times x^{n-1})\otimes1)=-(x^n\times x^n)\otimes b$, where $a$ is the image of the fundamental class of $S^1$.
Note that the cycle in $E^2_{2n,1}$ corresponding to the fundamental class is computed as
\[
	(x^n\times 1+x^{n-1}\times x+\cdots+1\times x^n)\otimes a.
\]

For the homotopy equalizer $\Hoeq(p^k_n,p^k_n)\to\mathbb{C}P^n$, we have the $E^2$-term as follows:
\[
	E^2_{p,q}(\Hoeq(p^k_n,p^k_n))\cong H_p(\mathbb{C}P^n;H_q(F))\cong
	\begin{cases}
		H^{2n-p}(\mathbb{C}P^n)								& q=0 \\
		H^{2n-p}(\mathbb{C}P^n)\otimes\mathbb{Z}\{a,b\}		& q=1.
	\end{cases}
\]
Then, the under the shriek map $\tilde{\Delta}^!\colon E^2(E^k_n\to\mathbb{C}P^n\times\mathbb{C}P^n)\to E^2(\Hoeq(p^k_n,p^k_n)\to\mathbb{C}P^n)$, the fundamental class is mapped to
\[
	(n+1)x^n\otimes a\in E^2_{0,1}.
\]
Comparing the spectral sequences through the map
\[
\xymatrix{
	\Hoeq(p^k_n,p^k_n) \ar[r]^-{\tilde{\Delta}} \ar[d]
		& E^k_n \ar[d] \\
	\mathbb{C}P^n \ar[r]_-{\Delta}
		& \mathbb{C}P^n\times\mathbb{C}P^n,
}
\]
we have
\[
	d^2x^{n-1}=kx^n\otimes a\in E^2_{0,1}(\Hoeq(p^k_n,p^k_n))
\]
Thus we have
\[
	H_1(\Hoeq(p^k_n,p^k_n))\cong\mathbb{Z}\oplus\mathbb{Z}/k\mathbb{Z}
\]
and the coincidence Reidemeister trace $\rho^{\mathbb{Z}}(p^k_n,p^k_n)$ lives in the torsion part $\mathbb{Z}/k\mathbb{Z}$.
More precisely, we obtain the following result.

\begin{thm}
The following conditions are equivalent:
\begin{enumerate}
\item
$n+1$ is not divisible by $k$,
\item
$\rho^{\mathbb{Z}}(p^k_n,p^k_n)\ne0$,
\item
$\tilde{N}^\mathbb{Z}(p^k_n,p^k_n)=1$.
\item
$N^\mathbb{Z}(p^k_n,p^k_n)=1$.
\end{enumerate}
\end{thm}

\begin{proof}
The equivalence of (1) and (2) follows from the above argument.
The equivalence with the conditions (3) follows from the fact that $\Hoeq(p^k_n,p^k_n)$ is path-connected.
For the condition (4), it is sufficient to compute the cohomology Serre spectral sequence for $p^k_n\colon E^k_n\to\mathbb{C}P^n$ and to apply Corollary \ref{cor_primary_obstruction}.
\end{proof}

This result is the same as \cite[Theorem 1.3]{MR2147735}.
The nontriviality of the self-coincidence of the Hopf fibration \cite[Theorem 1.1]{MR2147735} cannot be observed from the coincidence Reidemeister trace in ordinary homology.

\begin{rem}
In view of Theorem \ref{thm_transpose}, one can apply the Serre spectral sequence of $\Hoeq(f,g)\to M$ to compute the Reidemeister trace.
To do this, we need to extend the shriek maps in Section \ref{shriek map on ss} for general maps $X_1\to X_2$.
But this spectral sequence is not good by the following reason.
In the Serre spectral sequence of $\Delta\colon N\to N\times N$, the permanent cycle corresponding to the fundamental class lives in $E^r_{n,0}$.
Then even if its image of the shriek map vanishes, which lives in $E^r_{m-n,0}$, one cannot say the Reidemeister trace is trivial due to the extension problem.
In contrast to this, in the spectral sequence we considered, the corresponding image lives in $E^r_{p,q}$ for $p\le0$.
Thus we do not need to solve the extension problem to observe the triviality of the Reidemeister trace.
\end{rem}

\section{Relation with loop coproduct}
\label{loop coproduct}

Let $M$ be a smooth connected closed oriented manifold of dimension $m$.
We consider the self-coincidence of the identity map $\id_M$.
The corresponding homotopy equalizer is the free loop space $\Hoeq(\id_M,\id_M)=LM$.
Consider the following pullback diagram:
\[
\xymatrix{
	LM \ar[d]
		& LM\times_MLM \ar[l]_-{\tilde{\Delta}} \ar[d] \\
	M\times M
		& M, \ar[l]^-{\Delta}
}
\]
where $LM\to M\times M$ is the evaluation at $0$ and $1/2$ and $LM\times_MLM\to LM$ is the concatenation of loops.
Then, we have the shriek map
\[
	\tilde{\Delta}^!\colon H_\ast(LM)\to H_{\ast-m}(LM\times_MLM).
\]
The \textit{loop coproduct} \cite{MR2079373}
\[
	\Psi\colon H_\ast(LM)\to H_{\ast-m}(LM\times LM)
\]
is defined to be the composite of $\tilde{\Delta}^!$ and the homomorphism induced by the canonical inclusion $LM\times_MLM\to LM\times LM$.

Let us consider the commutative diagram
\[
\xymatrix{
	LM \ar[d]
		& LM\times_MLM \ar[l] \ar[r] \ar[d]
		& LM\times LM \ar[dl] \\
	M^{[0,1]}
		& LM \ar[l],
}
\]
where the vertical $LM\to M^{[0,1]}$ is given by
\[
	\gamma\mapsto(t\to\gamma(2t)),
\]
the horizontal $LM\times_MLM\to LM$ is the concatenation of loops, the vertical $LM\times_MLM\to LM$ and $LM\times LM\to LM$ are the projections onto the first factor, and the horizontal $LM\to M^{[0,1]}$ is the inclusion.
This induces the following diagram on homology:
\[
\xymatrix{
	H_m(LM) \ar[r] \ar[d]
		& H_0(LM\times_MLM) \ar[r] \ar[d]
		& H_0(LM\times LM) \ar[dl] \\
	H_m(M) \ar[r]_-{R(\id,\id)}
		& H_0(LM).
}
\]
For the section $s\colon M\to LM$ such that $s(x)$ is the constant loop at $x$, consider the class $s_\ast[M]\in H_m(LM)$.
The loop coproduct $\Psi s_\ast[M]$ is computed by Tamanoi \cite[Theorem 3.1]{MR2577666} as
\[
	\Psi s_\ast[M]=\chi(M)[\ast]\times[\ast],
\]
where $\chi(M)$ is the Euler characteristic of $M$ and $\ast$ is the $0$-cycle of the constant loop at the base point.
Note that the map $H_m(LM)\to H_m(M)$ is the same as the homomorphism induced from the evaluation $LM\to M$ at $0$.
Then the image of $s_\ast[M]$ under this map is $[M]$.
Then we have
\[
	\rho(\id_M,\id_M)=R(\id_M,\id_M)[M]=\chi(M)[\ast].
\]

\bibliographystyle{alpha}
\bibliography{2016reidemeister}
\end{document}